\documentclass[12pt]{amsart}

\usepackage{amssymb}
\usepackage{stmaryrd}
\usepackage{fullpage}
\usepackage{amscd}
\usepackage[all]{xy}
\usepackage{comment}
\usepackage{mathrsfs}
\usepackage{longtable}

\usepackage{tikz}
\usetikzlibrary{positioning, calc}

\usepackage{color}

\renewcommand\mathbb\mathbf

\newcommand{\sm}{{\operatorname{sm}}}
\newcommand{\dR}{{\operatorname{dR}}}
\newcommand{\alg}{{\operatorname{alg}}}

\DeclareMathOperator{\GSp}{{GSp}}

\newcommand{\BC}{{\operatorname{BC}}}

\newcommand{\defi}[1]{\emph{#1}}

\newcommand{\To}{\longrightarrow}

\newcommand{\C}{{\mathbb C}}
\newcommand{\F}{{\mathbb F}}
\newcommand{\G}{{\mathbb G}}

\newcommand{\PP}{{\mathbb P}}
\newcommand{\Q}{{\mathbb Q}}
\newcommand{\R}{{\mathbb R}}
\newcommand{\Z}{{\mathbb Z}}

\newcommand{\ksep}{{k^{\operatorname{sep}}}}

\newcommand\ps[1]{{\llbracket#1\rrbracket}}

\newcommand{\pp}{{\mathfrak p}}

\newcommand\td\widetilde

\newcommand\bS{{\mathbb{S}}}

\newcommand\bG{{\mathbb{G}}}

\newcommand{\calP}{{\mathcal P}}

\newcommand\scdot{\,\cdot\,}
\def\presuper#1#2%
  {\mathop{}%
   \mathopen{\vphantom{#2}}^{#1}%
   \kern-\scriptspace%
   #2}

\def\Q{\mathbb{Q}}
\def\C{\mathbb{C}}
\def\R{\mathbb{R}}

\def\Z{\mathbb{Z}}

\def\<{\ensuremath{\langle}}
\def\>{\ensuremath{\rangle}}

\newcommand{\inv}{{\operatorname{inv}}}

\DeclareMathOperator{\Lie}{Lie} \DeclareMathOperator{\Hom}{Hom}
 
\DeclareMathOperator{\Gal}{Gal} 
 \DeclareMathOperator{\Res}{Res}

 \DeclareMathOperator{\divv}{div}
\DeclareMathOperator{\ord}{ord} 
\DeclareMathOperator{\Div}{Div} \DeclareMathOperator{\Pic}{Pic}
\DeclareMathOperator{\Prin}{Prin}
\renewcommand{\div}{\operatorname{div}}
 
\DeclareMathOperator{\Spec}{Spec}

\DeclareMathOperator{\trop}{trop}
\DeclareMathOperator{\Log}{Log}

\DeclareMathOperator{\val}{val}

\newcommand{\Ab}{{\operatorname{Ab}}}

\newcommand{\an}{{\operatorname{an}}}

\newcommand{\tors}{{\operatorname{tors}}}
\newcommand{\red}{{\operatorname{red}}}

\newcommand\sptxt[1]{\quad\text{ #1 }\quad}

\newcommand{\surjects}{\twoheadrightarrow}
\newcommand{\injects}{\hookrightarrow}

\newcommand{\tensor}{\otimes}

\newcommand\djunion\amalg

\newcommand{\isomto}{\overset{\sim}{\To}}

\numberwithin{equation}{section}
\newtheorem{theorem}[equation]{Theorem}
\newtheorem{lemma}[equation]{Lemma}
\newtheorem{corollary}[equation]{Corollary}
\newtheorem{proposition}[equation]{Proposition}

\theoremstyle{definition}
\newtheorem{definition}[equation]{Definition}

\newtheorem{example}[equation]{Example}

\theoremstyle{remark}
\newtheorem{remark}[equation]{Remark}

\renewcommand{\arraystretch}{1.4}

\usepackage[
	backref,
	pdfauthor={David Brown},
]{hyperref}

\newcommand\bb[1]{\ensuremath{{\mathbb{#1}}}}

\newcommand\ms[1]{\ensuremath{{\mathscr{#1}}}}
\newcommand\mf[1]{\ensuremath{{\mathfrak{#1}}}}
\newcommand\fX{{\mf X}}

\newcommand\fU{{\mf U}}

\newcommand\fL{{\mf L}}
\newcommand\sH{{\ms H}}
\newcommand\sO{{\ms O}}
\newcommand\bB{{\bb B}}
\newcommand\sA{{\ms A}}
\newcommand\sM{{\ms M}}
\newcommand\fp{{\mf p}}

\newcommand\Djunion{\coprod}
\newcommand\Dsum{\bigoplus}
\newcommand\angles[1]{\ensuremath{\langle#1\rangle}}
\DeclareSymbolFont{symbolsC}{U}{pxsyc}{m}{n}
\DeclareMathSymbol{\coloneq}{\mathrel}{symbolsC}{66}

\newcommand\BCint{\presuper\BC\int}
\newcommand\Abint{\presuper\Ab\int}

\newcommand{\Supp}{\operatorname{Supp}}

\renewcommand\bb[1]{\ensuremath{{\mathbf{#1}}}}

\renewcommand\G{\bb G}
\renewcommand\R{\bb R}
\renewcommand\C{\bb C}
\renewcommand\Z{\bb Z}
\renewcommand\Q{\bb Q}
\renewcommand\F{\bb F}

\begin{document}


\title{Uniform bounds for the number of rational points on curves of small Mordell--Weil rank}
\author{Eric Katz}
\address{Department of Mathematics, The Ohio State University, Columbus, Ohio, USA}
\email{katz.60@osu.edu}
\author{Joseph Rabinoff}
\address{Department of Mathematics, Georgia Tech, Atlanta, Georgia, USA}
\email{rabinoff@math.gatech.edu}
\author{David Zureick-Brown}
\address{Department of Mathematics and Computer Science, Emory University,
Atlanta, Georgia, USA}
\email{dzb@mathcs.emory.edu}

\date{}
\thanks{}

\begin{abstract}
  Let $X$ be a curve of genus $g\geq 2$ over a number field $F$ of degree
  $d = [F:\Q]$.  The conjectural existence of a uniform bound $N(g,d)$ on the
  number $\#X(F)$ of $F$-rational points of $X$ is an outstanding open problem
  in arithmetic geometry, known by Caporaso, Harris, and Mazur to follow from
  the Bombieri--Lang conjecture.  A related conjecture posits the existence of a
  uniform bound $N_{\tors,\dagger}(g,d)$ on the number of geometric torsion
  points of the Jacobian $J$ of $X$ which lie on the image of $X$ under an
  Abel--Jacobi map.  For fixed $X$ this quantity was conjectured to be finite by
  the Manin--Mumford conjecture, and was proved to be so by
  Raynaud.

  We give an explicit uniform bound on $\#X(F)$ when $X$ has Mordell--Weil rank
  $r\leq g-3$.  This generalizes recent work of Stoll on uniform bounds for
  \emph{hyperelliptic} curves of small rank to arbitrary curves.  Using the same
  techniques, we give an explicit, \emph{unconditional} uniform bound on the
  number of $F$-rational torsion points of $J$ lying on the image of $X$ under
  an Abel--Jacobi map.  We also give an explicit uniform bound on the number of
  geometric torsion points of $J$ lying on $X$ when the reduction type of $X$ is
  highly degenerate.

  Our methods combine Chabauty--Coleman's $p$-adic integration,
  non-Archimedean potential theory on Berkovich curves, and the theory of
  linear systems and divisors on metric graphs.
\end{abstract}

\dedicatory{Dedicated to the memory of Robert Coleman.}

\maketitle

\section{Introduction}

Let $X$ be a curve of genus $g\geq 2$ over a number field $F$ of degree $d = [F:\Q]$.  The conjectural
existence of a uniform bound $N(g,d)$ on
the number $\#X(F)$ of $F$-rational points of $X$ is an outstanding open problem in
arithmetic geometry, known by \cite{CaporasoHM:uniformity} to follow from the
Bombieri--Lang
conjecture.
A related conjecture posits the existence of a uniform bound
$N_{\tors,\dagger}(g,d)$ on the number of geometric torsion points of the
Jacobian of $X$ which lie on the image of $X$ under an Abel--Jacobi map.  For
fixed $X$ this quantity was conjectured to be finite by the Manin--Mumford conjecture, and was
proved to be so by Raynaud~\cite{raynaud:maninMumford}.

In this paper we obtain both kinds of uniform bounds for large classes of curves
where uniformity was previously unknown.  To do so, we combine Chabauty and
Coleman's method of $p$-adic integration, potential theory on
Berkovich curves, and the theory of linear systems and divisors on metric
graphs.  The main theorems are as follows.

\begin{theorem}
\label{T:uniformity-K-points}
Let $d\geq 1$ and $g\geq 3$ be integers.  There exists an explicit constant
$N(g,d)$ such that for any number field $F$ of degree $d$ and any smooth,
proper, geometrically connected genus $g$ curve $X / F$ of Mordell--Weil rank at
most $g-3$, we have
\[
\#X(F) \leq N(g,d).
\]
  
\end{theorem}

The Mordell--Weil rank is by definition the rank of the finitely generated
abelian group $J(F)$, where $J$ is the Jacobian of $X$.
Theorem~\ref{T:uniformity-K-points} is an improvement on a theorem of
Stoll~\cite{Stoll:uniformity}, which applies to \emph{hyperelliptic} $X$.  The
methods used to prove Theorem~\ref{T:uniformity-K-points} are largely inspired
by Stoll's ideas.  (See Section~\ref{sec:overview} for a discussion of Stoll's results
and their relation to this paper.)

There are any number of different ways of expressing the bound $N(g,d)$.  For
instance, in the case $F=\Q$, we can take
\[ N(g,1) = 84g^2 - 98g + 28 \]
by applying Theorem~\ref{thm:with.ss.model} with $K = \Q_3$ and by
using~\eqref{eq:nice.NpAB.bound}. 

Next, we define an equivalence relation on the set of $\bar F$-points of a curve
$X/F$ as follows: we say that two points $P,Q$ are equivalent if $mP$ is
linearly equivalent to $mQ$ on $X_{\bar F}$ for some integer $m\geq 1$.  We
define a \defi{torsion packet} to be an equivalence class under this relation.
Equivalently, a torsion packet is the inverse image of the group of geometric
torsion points of the Jacobian $J$ of $X$ under an Abel--Jacobi map
$X_{\bar F}\injects J_{\bar F}$.  Replacing $\bar F$ with $F$, one has a notion of a
\defi{rational torsion packet} as well.  As mentioned above, Raynaud
\cite{raynaud:maninMumford} proved that every torsion packet of a curve is
finite.  Many additional proofs, with an assortment of techniques and
generalizations, were given later by
\cite{buium:p-adic-jets, coleman:ramified_torsion_curves,
  hindry:autour-conjecture-manin-mumford, ullmo:positivite-manin-mumford,
  pilazannier}, and others.
Several of these proofs rely on $p$-adic methods, with the method of Coleman
being particularly closely related to ours.

A \emph{uniform} bound on the size of the torsion packets of a curve of genus
$g\geq 2$ is expected but still conjectural.  We offer two results in this
direction.  The first unconditional result concerns rational torsion packets
and is proved along with Theorem~\ref{T:uniformity-K-points}.  To our knowledge,
no uniformity result was previously known even in this case for general curves
(for hyperelliptic curves it follows from~\cite[Theorem~8.1]{Stoll:uniformity}).

\begin{theorem}
\label{T:uniformity-K-torsion}
Let $d\geq 1$ and $g\geq 3$ be integers.  There exists an explicit constant
$N_{\tors}(g,d)$ such that for any number field $F$ of degree $d$, any
smooth, proper, geometrically connected genus $g$ curve $X/F$, and any
Abel--Jacobi embedding $\iota\colon X \hookrightarrow J$ into its Jacobian
(defined over $F$), we have
\[
\#\iota^{-1}\big(J(F)_{\tors}\big) \leq N_{\tors}(g,d).
\]
\end{theorem}

In fact, one may take $N_\tors(g,d) = N(g,d)$, the same constant in
Theorem~\ref{T:uniformity-K-points}.  Note that here there is no restriction on
the Mordell--Weil rank.

The second result concerns (geometric) torsion packets.  It involves the
following restriction on the reduction type.  Let $F$ be a number field, and let
$\fp$ be a finite prime of $F$.  Let $X$ be a smooth, proper, geometrically
connected curve of genus $g\geq 2$ over $F$.  Let $\fX$ be the stable model of
$X$ over an algebraic closure of $F_\pp$.  For each irreducible component $C$ of
the special fiber $\fX_s$ of $\fX$, let $g(C)$ denote its geometric genus, and let
$n_C$ denote the number points of the normalization of $C$ mapping to nodal
points of $\fX_s$. We say that $X$ \defi{satisfies condition} ($\dagger$)
\defi{at $\fp$} provided that
\begin{equation}
  \label{eq:reduction-condition}
  \tag{$\dag$}
  g > 2g(C) + n_C
\end{equation}
for each component $C$ of $\fX_s$.

\begin{theorem}
\label{T:uniformity-torsion}
Let $d\geq 1$ and $g\geq 4$ be integers.  There exists an explicit constant
$N_{\tors,\dagger}(g,d)$ such that for any number field $F$ of degree $d$ and
any smooth, proper, geometrically connected genus $g$ curve $X/F$ which
satisfies condition \textup{($\dagger$)} at some prime $\fp$ of $F$, we have
\[
\#\iota^{-1}\big(J(\bar F)_{\tors}\big) \leq N_{\tors,\dagger}(g,d).
\]  
for any Abel--Jacobi embedding
$\iota\colon X_{\bar F}\injects J_{\bar F}$ of $X_{\bar F}$ into its Jacobian.
\end{theorem}

The condition ($\dagger$) is satisfied at $\fp$, for instance, when $X$ has totally
degenerate trivalent stable reduction over $F_\fp$.  One can take
\[ N_{\tors,\dagger}(g,d) = (16g^2-12g)\,N_2\big((4d\cdot 7^{2g^2+g+1})^{-1},\,2g-2\big), \]
where 
\begin{equation}\label{eq:explicit.geom.bound}
  N_2(s,N_0) = \min\big\{N\in\Z_{\geq 1}~:~s(n-N_0)>\lfloor \log_2(n) \rfloor ~\forall
  n\geq N\big\}. 
\end{equation}
See Theorem~\ref{thm:geom.torsion.bound} for a more precise statement.

A uniform bound as in Theorem~\ref{T:uniformity-torsion} for the size of
geometric torsion packets was previously known (see~\cite{buium:p-adic-jets}) for
curves of good reduction at a fixed prime $\pp$.  This result uses work of
Coleman~\cite{coleman:ramified_torsion_curves}, who also deduces uniform bounds
in many situations, still in the good reduction case: for instance, if $X/\Q$ has
ordinary good reduction at $p$ and its Jacobian $J$ has potential complex multiplication, then
$\#\iota^{-1}\big(J(\bar\Q)_{\tors}\big) \leq gp$.
Theorem~\ref{T:uniformity-torsion}, on the other hand, applies to curves with
highly \emph{degenerate} reduction, and hence approaches the uniform Manin--Mumford
conjecture from the other extreme.  It is also independent of the residue
characteristic of $F_\fp$.

The full power of the general machinery developed in this paper is needed for
the proof of Theorem~\ref{T:uniformity-torsion}, which is striking in that it
uses $p$-adic integration techniques to bound the number of \emph{geometric}
torsion points.  Whereas Theorems~\ref{T:uniformity-K-points}
and~\ref{T:uniformity-K-torsion} only involve integration on discs and annuli,
which was Stoll's idea, Theorem~\ref{T:uniformity-torsion} requires integrating
over finitely many \emph{wide open subdomains} which cover $X^{\an}$, and, as
such, is more subtle.  (See Section~\ref{sec:overview} below for a more detailed
summary of the proofs.)

\begin{remark}
\label{R:Bhargavology}
One expects that the Mordell--Weil rank of the Jacobian of a curve is usually 0 or 1. In practice one needs a family of curves over a rational base to even make this precise. One therefore often restricts to families of hyperelliptic (or sometimes low genus plane) curves, in which case there are very recent partial results: see~\cite{bhargava2013average} (elliptic curves), \cite{bhargava2012average} (Jacobians of hyperelliptic curves), and \cite{Thorne:E6-arithmetic}  (certain plane quartics).
Combining these rank results with Chabauty's method and other techniques,
several recent results prove that the uniformity conjecture holds for a random
curve (in that there are no ``nonobvious'' points): see \cite{PoonenS:mostOddDegree,bhargava2013most,shankar2013average,bhargava2013pencils} (see \cite{ho:howManySurvey} for a recent survey).

\end{remark}
\subsection{Overview of the proofs}\label{sec:overview}

Our central technique is Chabauty and Coleman's method of $p$-adic
integration. In a 1941 paper, Chabauty \cite{chabauty41-sur-les-points} proved
the Mordell conjecture in the special case of curves with Mordell--Weil rank at
most $g-1$, via a study of the $p$-adic Lie theory of the Jacobian of $X$. Four
decades later, Coleman \cite{Coleman:effectiveChabauty} made Chabauty's method
\emph{explicit}: he proved that for a curve $X/\Q$ of genus $g\geq 2$, rank
$r < g$, and a prime $p > 2g$ of good reduction,
$$\#X(\Q) \leq \#X(\F_{p}) + 2g - 2.$$
Coleman's method has been refined by many authors: these authors
\cite{LorenziniT:thue, McCallumP:chabautySurvey, stoll06:rational_points,
  katz_dzb13:chabauty_coleman}
allow $X$ to have bad reduction at $p$ and improve the $2g-2$ to $2r$,
\cite{siksek:symmetricPowerChabauty, park:tropicalSymmetricChabauty} generalize
to symmetric powers of curves, and a large body of work by many authors allow
one to explicitly execute this method in Magma for any particular curve of low
genus and low rank, frequently allowing one to compute $X(\Q)$ exactly.

Our starting point for proving Theorems~\ref{T:uniformity-K-points}
and~\ref{T:uniformity-K-torsion} is the recent progress of
Stoll~\cite{Stoll:uniformity}, who proves that for any \emph{hyperelliptic}
curve $X/\Q$ with Jacobian of rank $r \leq g-3$,
$$\#X(\Q) \leq 8(r+4)(g-1) + \max\{1,4r\} \cdot g.$$
While this bound still depends on $r$ and $g$, its independence from $p$ is a
substantial improvement.  This improvement is made possible by fixing a prime
$p$ (generally small and odd) and considering curves $X/\Q_p$ with arbitrary
reduction type.  Stoll's bold idea is to decompose $X(\Q_p)$ into a disjoint
union of residue discs and residue \emph{annuli} and to execute Chabauty's
method on \emph{both}.  The decomposition is achieved by performing a careful
analysis of the minimal regular model of $X$ over $\Z_p$.  Bounding zeros of
integrals on annuli is somewhat subtle: monodromy becomes an issue, and a key
technical feature of Stoll's work is his analysis and comparison of analytic
continuation and the emergent $p$-adic logarithms.  Stoll's method exploits the
description of differentials on a hyperelliptic curve as $f(x)dx/y$;
using an explicit calculation,
he is able to analyze the zeroes of the resulting integral directly via
Newton polygons.

In contrast, a differential on a typical curve may lack an explicit description,
and a direct, explicit analysis is impervious to classical
methods.
Moreover, one cannot hope to attain any kind of geometric bound as in
Theorem~\ref{T:uniformity-torsion} by analyzing $p$-adic integrals on discs and
annuli alone, as the antiderivative of an analytic function on a disc or annulus
may well have infinitely many geometric zeros.  This is where potential theory
on Berkovich analytic curves and the theory of linear systems on metric graphs
becomes useful.  To be clear, the inputs into the proofs of
Theorems~\ref{T:uniformity-K-points} and~\ref{T:uniformity-K-torsion} are
Stoll's bounds~\cite[Proposition~5.3]{Stoll:uniformity} on the number of discs
and annuli covering $X(\Q_p)$, and a new method of bounding the zeros of an
integral on an open annulus (Corollary~\ref{cor:actual.annulus.bound}).  As
mentioned before, the full power of the general machinery developed in this
paper is needed for the proof of Theorem~\ref{T:uniformity-torsion}; only a
fraction of it (namely, Section~\ref{sec:berkovich-curves} along with
Lemma~\ref{lem:unstable.combinatorics}) is needed for
Corollary~\ref{cor:actual.annulus.bound}.

Let us give an overview of our methods.  They are entirely geometric, so we work
over the field $\C_p$, the completion of an algebraic closure of $\Q_p$.  Let
$X$ be a curve over $\C_p$ of genus $g\geq 2$, and let $X^\an$ denote the
analytification of $X$, in the sense of
Berkovich~\cite{berkovic90:analytic_geometry}.  This is a reasonable topological
space in that it deformation retracts onto a finite metric graph
$\Gamma\subset X^\an$ called a \defi{skeleton}, whose combinatorics is
controlled by a semistable model of $X$.  (As $\C_p$ is algebraically closed,
such a model exists.)  If $f$ is a nonzero rational function on $X$, then
$-\log|f|$ is a piecewise affine function on $\Gamma$ with integer slopes.
Letting $\tau \colon X^\an\to\Gamma$ denote the deformation retraction, the
inverse image $\tau^{-1}(V)$ of a small neighborhood $V$ of a vertex $v$ of
$\Gamma$ is a \defi{basic wide open subdomain} in the sense of
Coleman~\cite[Section~3]{ColemanReciprocity}.  One can cover $X^\an$ by finitely many
basic wide open subdomains.

Our proof (roughly) proceeds by using the following steps.
\begin{enumerate}
\item Let $f$ be a nonzero analytic function on a basic wide open $U$ with central
  vertex $v$.  A basic fact from potential theory on $X^\an$ implies
  that $\deg(\div(f))$ can be calculated by summing the slopes of $-\log|f|$
  along the incoming edges at $v$ (see Proposition~\ref{prop:zeros.on.wide.open}).

\item Let $\omega$ be an exact differential form on $U$, and let $f = \int\omega$
  be an antiderivative.  A Newton polygon calculation
  (Proposition~\ref{p:annularbounds}) relates the slopes of $-\log|f|$ with the
  slopes of $-\log\|\omega\|$.  Here $\|\omega\|$ is the norm of $\omega$ with
  respect to the canonical metric on $\Omega^1_{X/\C_p}$, described
  in Section~\ref{sec:rosenlicht}.  The ``error term'' $N_p(\scdot,\scdot)$ appearing
  in~\eqref{eq:explicit.geom.bound} is introduced at this point.

\item Suppose now that $\omega$ is a global differential form on $X$.
  Then the restriction $F$ of $-\log\|\omega\|$ to $\Gamma$ is a ``section of
  the tropical canonical sheaf,'' in that $\div(F) + K_\Gamma\geq 0$, where
  $K_\Gamma$ is the canonical divisor on the graph $\Gamma$.  This is a
  consequence of the slope formula (otherwise known as the Poincar\'e--Lelong
  formula) for line bundles on Berkovich curves, which we prove in
  Theorem~\ref{thm:slope.formula.2}.

\item With $\omega$ and $F$ as above, we use a combinatorial argument
  (Lemma~\ref{lem:unstable.combinatorics}) about linear systems on vertex-weighted metric
  graphs to bound the slopes of $F$ in terms of the genus of the graph $\Gamma$,
  which is bounded by the genus of the curve.  This step plays the role of the
  usual Riemann--Roch part of the Chabauty--Coleman argument.  It also plays the
  role of~\cite[Corollary~6.7]{Stoll:uniformity}, which is proved
  using explicit calculations on hyperelliptic curves.

\item Using Coleman's calculation of the de Rham cohomology of a wide open
  subdomain $U$, under the restriction ($\dagger$) we can produce a nonzero
  global differential form $\omega$ which is exact on $U$.  Combining the above
  steps then provides a uniform bound on the number of zeros of $\int\omega$ on
  $U$.  Covering $X^\an$ by such wide opens $U$ yields
  Theorem~\ref{T:uniformity-torsion}, as the integral of any differential form
  vanishes on torsion points.

\item An open annulus is a simple kind of wide open subdomain.  Specializing the
  above results to annuli (Corollary~\ref{cor:actual.annulus.bound}) gives the
  generalization of~\cite[Proposition~7.7]{Stoll:uniformity} needed to prove
  Theorems~\ref{T:uniformity-K-points} and~\ref{T:uniformity-K-torsion} by
  using~\cite[Proposition~5.3]{Stoll:uniformity}. 
\end{enumerate}
It should be mentioned that in principle one can avoid the Berkovich language by
using intersection theory on semistable curves, but this leads to fussy
arguments and frequent base changes and at certain points is very difficult to
do.  We hope the reader will agree that the analytic framework is much more
natural.

In the summary above we have suppressed a major technical difficulty.  By an
``antiderivative'' of $\omega$, we always mean an analytic function $f$ such
that $df = \omega$.  The definite integral $\int_x^y f$ is then defined to be
$f(y) - f(x)$; this is what is needed for Newton polygons and potential theory.
However, for curves of bad reduction this does \emph{not} generally coincide
with the abelian integration used in the Chabauty--Coleman method, defined in terms of a
$p$-adic logarithm on the Jacobian.  Indeed, the former kind of integration will
have $p$-adic periods, whereas the latter cannot.  This was realized by
Stoll~\cite{Stoll:uniformity}, who found a way to compare the integrals on
annuli.  A systematic comparison between these integration theories in general,
given in Section~\ref{sec:integration}, should be of independent interest.

\subsection{Organization of paper}

In Section~\ref{sec:berkovich-curves}, we recall several basic facts about Berkovich
curves, and we develop the $p$-adic analytic machinery that we will need.  The
main features are the following: Theorem~\ref{thm:slope.formula.2}, a generalization of the
slope formula~\cite[Theorem~5.15]{baker_payne_rabinoff13:analytic_curves} to
sections of formally metrized line bundles; a careful treatment of Rosenlicht
differentials, a generalization of the relative dualizing sheaf to a semistable
curve over a possibly nondiscretely valued field, needed in order to define the
norm $\|\omega\|$ of a differential; and Coleman's calculation
(Theorem~\ref{thm:H1dR}) of the de Rham cohomology of a basic wide open
subdomain.

In Section~\ref{sec:integration}, we recall the basic properties of the
Berkovich--Coleman integral and the abelian integral in our somewhat restricted setting.
We then prove a result (Proposition~\ref{prop:correction.wide.open}) comparing
the two: essentially, the difference is controlled by the \emph{tropical}
Abel--Jacobi map.  The non-Archimedean uniformization theory of abelian
varieties plays a central role here.

In Section~\ref{sec:bounds-zeroes}, we explicitly bound the slopes of an analytic function
$f$ on an annulus in terms of the slopes of $\omega = df$
(Proposition~\ref{p:annularbounds}) and deduce, via a quick combinatorial
argument (Lemma~\ref{lem:graph.combinatorics}), a bound on the number
of zeroes of the integral of an exact differential on a wide open
(Theorem~\ref{thm:wideopen.bounds}).

Finally, in Section~\ref{sec:uniform-bounds}, we put everything together, proving our
main theorems on uniform bounds.

\section{Berkovich curves}
\label{sec:berkovich-curves}

In this section we develop the basic geometric facts about analytic curves over
non-Archimedean fields that will be used below.

\subsection{General notation}
We will use the following notations for non-Archimedean fields, in this section
only.  In subsequent sections we will generally restrict our attention to $\C_p$.
\smallskip

\def\arraystretch{1.0}
\begin{tabular}{ll}
  $K$ & A field that is complete with respect to a nontrivial,
  non-Archimedean valuation. \\
  $\val$ & $:K^\times\to\R\cup\{\infty\}$, the fixed valuation on $K$. \\
  $|\scdot|$ & $= \exp(-\val(\scdot))$, a corresponding absolute value. \\
  $R$ & $= \sO_K$, the valuation ring of $K$. \\
  $k$ & The residue field of $K$. \\
  $\Lambda$ & $= \val(K^\times)\subset\R$, the value group of $K$. \\
  $\sqrt\Lambda$ & The saturation of $\Lambda$.
\end{tabular}

\smallskip
Let $X$ be a proper $K$-scheme, and let $\fX$ be a proper, flat $R$-model of
$X$.  We use the following notations:
\smallskip

\begin{tabular}{ll}
  $X^\an$ & The analytification of $X$, in the sense of Berkovich~\cite{berkovic90:analytic_geometry}. \\
  $\sH(x)$ & The completed residue field at a point $x\in X^\an$. \\
  $\fX_k$ & The special fiber of $\fX$. \\
  $\red$ & $:X^\an\to\fX_k$, the reduction or specialization map.
\end{tabular}

\smallskip The completed residue field is a valued field extension of $K$.  For
$x\in X^\an$ the reduction $\red(x)$ is defined by applying the valuative
criterion of properness to the canonical $K$-morphism $\Spec(\sH(x))\to\fX$.  The
reduction map is anticontinuous, in that the inverse image of a closed set is
open.

\subsection{Skeletons}\label{sec:skeletons}
Here we fix our notions regarding non-Archimedean analytic curves and their
skeletons.  We adhere closely to the treatment
in~\cite{baker_payne_rabinoff13:analytic_curves}, our primary reference.

Let $X$ be a smooth, proper, geometrically connected $K$-curve.  We say that a
semistable $R$-model $\fX$ is \defi{split} if the $G_k$-action on the dual graph
of $\fX_\ksep$ is trivial, where $G_k = \Gal(\ksep/k)$.  Equivalently, we
require that each component of $\fX_\ksep$ be defined over $k$, that all nodes
of $\fX_\ksep$ be $k$-rational, and that the completed local ring of $\fX_\ksep$
at a node be isomorphic to $k\ps{R,S}/(RS)$.  (The final condition rules out the
possibility that $G_k$ acts via an involution on a loop edge, i.e., that it
interchanges ``tangent directions'' at the node.)  Let $\fX$ be a split
semistable $R$-model of $X$.  We will use the following notations for the
structure theory of $X^\an$: \smallskip

\begin{tabular}{ll}
  $\Gamma_\fX$ & $\subset X^\an$, the skeleton associated to $\fX$. \\
  $\tau$ & $:X^\an\to\Gamma_\fX$, the retraction to the skeleton. \\
  $g(x)$ & The genus of a type-$2$ point $x\in X^\an$.
\end{tabular}

\smallskip In general, a \defi{skeleton} of $X$ is a skeleton
$\Gamma = \Gamma_\fX$ corresponding to some split semistable model $\fX$ of $X$.
It is a $\Lambda$-metric graph (a metric graph with edge lengths in
$\Lambda$) whose vertices correspond bijectively to the generic points of
$\fX_k$ in the following way: if $x\in\Gamma_\fX$ is a vertex, then $\zeta = \red(x)$ is a
generic point of $\fX_k$, and $\red^{-1}(\zeta) = \{x\}$.  The edges of
$\Gamma_\fX$ correspond to the singular points of $\fX_k$, as follows.  For
$\varpi\in R$ nonzero we let
\[ \bS(\varpi)_+ = \big\{ \xi\in\bG_m^\an~:~|\varpi| < |T(\xi)| < 1 \big\}, \]
the open annulus of modulus $|\varpi|$.  Here $T$ is a parameter on $\bG_m$;
that is, $\bG_m = \Spec(K[T,T^{-1}])$.  If $\td x\in\fX_k$ is a node, then
$\red^{-1}(\td x)\cong\bS(\varpi)_+$ for some $\varpi\in R$ with
$|\varpi|\in(0,1)$; the open edge $e$ of $\Gamma_\fX$ corresponding to $\td x$
is the skeleton of the annulus $\bS(\varpi)_+$ (see
Section~\ref{sec:slopes.on.annuli}), and the length of $e$ is the logarithmic
modulus $\val(\varpi)\in\Lambda$ of $\bS(\varpi)_+$, which is an isomorphism
invariant.

The \defi{weight} $g(x)$ of a vertex $x\in\Gamma_\fX$ is defined to be the genus
of the type-$2$ point $x\in X^\an$, which in turn is the geometric genus of the
corresponding component of $\fX_k$. 
We have the basic identity
\begin{equation}\label{eq:genus.formula}
  g(X) = h_1(\Gamma_\fX) + \sum_{x\in\Gamma_\fX} g(x),
\end{equation}
where $g(X)$ is the genus of the curve $X$, and
$h_1(\Gamma_\fX) = \dim_\Q H_1(\Gamma_\fX,\Q)$ is the first (singular) Betti number
of $\Gamma_\fX$.  

Any curve admits a split semistable model (and hence a skeleton) after
potentially making a finite extension of the ground field $K$, of degree bounded
by the genus (see the proof of Theorem~\ref{thm:geom.torsion.bound}).  If $X$ has a
skeleton, then it has a minimal skeleton, which comes from a stable $R$-model
$\fX$ of $X$.  If $g\geq 2$, then the minimal skeleton is unique, and we denote
it by $\Gamma_{\min}$.  The vertices of $\Gamma_{\min}$ are the points
of $X^\an$ of nonzero genus and the points of $\Gamma_{\min}$ of valency greater
than $2$.

\begin{remark}\label{rem:split.condition}
  \newcommand\hatbarK{{\widehat{\overline{K}}}} Let $X$ be a $K$-curve as above,
  let $\C_K = \hatbarK$ be the completion of the algebraic closure of $K$, and
  let $X' = X_{\C_K}$ be the base change.  If $\fX$ is a semistable model of $X$,
  then the base change $\fX'$ to the ring of integers in $\C_K$ is a
  semistable model, which is necessarily split as the residue field of
  $\C_K$ is algebraically closed.  The original model $\fX$ is split if and
  only if the natural action of $G_K=\Gal(K^{\operatorname{sep}}/K)$ on $X'^\an$
  fixes $\Gamma_{\fX'}\subset X'^\an$ pointwise, that is, if the skeleton
  $\Gamma_\fX$ is ``defined over $K$.''  Indeed,
  Berkovich~\cite{berkovic90:analytic_geometry} defines the skeleton associated
  to a nonsplit semistable model as the quotient of $\Gamma_\fX$ by the action
  of $G_K$.  The split condition is necessary for the formal fibers over nodes
  in $\fX_k$ to be $K$-isomorphic to open annuli, which we use repeatedly.
\end{remark}

\begin{remark}\label{rem:comb.types}
  Suppose that the genus $g$ of $X$ is at least $2$, let
  $\Gamma = \Gamma_{\min}$ be the minimal skeleton, and let $G$ be the
  underlying vertex-weighted (nonmetric) graph.  Then $G$ is a connected graph
  of genus $g$ with the property that any vertex of valency $1$ or $2$ has
  positive weight.  It is easy to see that there are finitely many isomorphism
  classes of such graphs.  In other words, for fixed $g$, there are finitely
  many \defi{combinatorial types} of minimal skeletons of curves of genus $g$.
  This crucial observation allows us to derive uniform bounds from
  stable models (see Section~\ref{sec:stable.graphs} for much more precise
  statements).
\end{remark}

\subsection{Metrized line bundles and the slope formula}\label{sec:slope.formula}
In this section we assume that our non-Archimedean field $K$ is algebraically
closed, which implies that $k$ is algebraically closed and
$\Lambda = \sqrt\Lambda$.  Let $X$ be a curve as in Section~\ref{sec:skeletons}, and let
$\Gamma\subset X^\an$ be a skeleton which is not a point.  There is a
well-developed theory of divisors and linear equivalence on graphs and metric
graphs, which we briefly recall here (see~\cite{Baker:specialization} and the
references therein for details).  A \defi{tropical meromorphic function} on
$\Gamma$ is a continuous, piecewise affine-linear function $F\colon\Gamma\to\R$ with
integral slopes.  A \defi{divisor} on $\Gamma$ is a formal sum of points of
$\Gamma$; the group of divisors is denoted $\Div(\Gamma)$.  The divisor of a
meromorphic function $F$ is $\div(F) = \sum_{x\in\Gamma}\ord_x(F)\,(x)$, where
$\ord_x(F) = -\sum_{v\in T_x(\Gamma)} d_v F(x)$, $T_x(\Gamma)$ is the set of
tangent directions at $x$, and $d_v F(x)$ is the slope of $F$ in the direction
$v$.  In other words, $\ord_x(F)$ is the sum of the incoming slopes of
$F$ at $x$.

To reduce questions about curves to questions about skeletons, we will
need to relate divisors on $X$ to divisors on $\Gamma$.  The retraction map
$\tau\colon X^\an\to\Gamma$ extends by linearity to a map on divisors
\[ \tau_* \colon \Div(X) \To \Div(\Gamma). \]

\begin{theorem} \label{thm:slope.formula}
  Let $f$ be a nonzero meromorphic function on $X$, and let
  $F = -\log|f|\big|_\Gamma$.  Then $F$ is a tropical meromorphic function on
  $\Gamma$ and 
  \[ \div(F) = \tau_* \div(f). \]
\end{theorem}

\begin{proof}
  This is a consequence of the slope formula for non-Archimedean
  curves (see~\cite[Theorem~5.15]{baker_payne_rabinoff13:analytic_curves}).
\end{proof}

We will need a generalization of Theorem~\ref{thm:slope.formula} that applies to
a meromorphic section of a formally metrized line bundle.
Theorem~\ref{thm:slope.formula.2} below can
in principle be extracted from Thuillier's
Poincar\'e--Lelong
formula~\cite[Proposition~4.2.3]{thuillier05:thesis},
and indeed should be seen as a reformulation of~\cite[Proposition~4.2.3]{thuillier05:thesis}, but it
is easier to derive it from the slope formula as it appears
in~\cite[Theorem~5.15]{baker_payne_rabinoff13:analytic_curves}.
In the discretely valued case, a version of Theorem~\ref{thm:slope.formula.2} can be found in
Christensen's thesis~\cite[Satz~1.3]{christensen:thesis}, with a similar proof.

The formal metric on a line bundle with an integral model is a basic
construction in Arakelov theory, which we briefly recall.  Let $\fX$ be an
admissible formal $R$-scheme in the sense
of~\cite{bosch_lutkeboh93:formal_rigid_geometry_I}, that is, a flat formal
$R$-scheme of topological finite presentation.  Let $X = \fX_\eta$ be the
analytic generic fiber, a $K$-analytic space.  Let $\fL$ be a line bundle on
$\fX$, and let $L = \fL_\eta$, a line bundle on $X$.  Let $s$ be a nonzero
meromorphic section of $L$, and let $x\in X$ be a point which is not a pole of
$s$.  Let $\fU\subset\fX$ be an open neighborhood of $\red(x)\in\fX$ on which
$\fL$ is trivial.  Then $U = \red^{-1}(\fU) = \fU_\eta$ is a closed analytic
domain containing $x$ on which $L$ is trivial, so we can write $s|_U = ft$,
where $t$ is a nonvanishing section of $\fL|_\fU$ and $f$ is a nonzero
meromorphic function on $U$.  The \defi{formal metric} on $L$ induced by $\fL$
is the metric $\|\scdot\|_\fL$ defined by
\[ \|s(x)\|_\fL \coloneq |f(x)|. \]
This is independent of all choices because an invertible function on $\fU$ has
absolute value $1$ everywhere.

In the algebraic situation, let $\fX$ be a \emph{proper} and flat $R$-scheme
with generic fiber $X$, and let $\hat\fX$ denote the completion with respect to
an ideal of definition in $R$.  Then $\hat\fX$ is a proper
admissible formal $R$-scheme, and there is a canonical isomorphism
$X^\an\cong\hat\fX_\eta$.  Hence any line bundle $\fL$ on $\fX$ with generic
fiber $L$ induces a formal metric $\|\scdot\|_\fL$ on $L^\an$.

\begin{remark} \label{rem:formal.metric.intersections}
  Formal metrics have the following intersection-theoretic interpretation over a
  discretely valued field $K$. (Note that the definition of $\|\scdot\|_\fL$
  above does not use that $K$ is algebraically closed.)  Suppose that $\Z$ is
  the value group of $K$.  For simplicity we restrict ourselves to a regular
  split semistable model $\fX$ of a smooth, proper, geometrically connected
  curve $X$.  A meromorphic section $s$ of $L$ can be regarded as a meromorphic
  section of $\fL$, and hence has an order of vanishing $\ord_D(s)$ along any
  irreducible component $D$ of $\fX_k$.  If $\zeta\in X^\an$ is the point
  reducing to the generic point of $D$, then we have the equality
  \[ -\log\|s(\zeta)\|_\fL = \ord_D(s). \]
  This follows from the observation that $\ord_D\colon K(X)^\times\to\Z$ reduces
  to (i.e., is centered at) the generic point of $D$.
\end{remark}

\begin{theorem}[The slope formula]\label{thm:slope.formula.2} 
  Let $X$ be a smooth, proper, connected $K$-curve, and let $\fX$ be a semistable
  $R$-model of $X$ with corresponding skeleton $\Gamma_\fX\subset X^\an$.
  Assume that $\fX$ is not smooth, so that $\Gamma_\fX$ is not a point. Let
  $\fL$ be a line bundle on $\fX$, let $L = \fL|_X$, and let $s$ be a nonzero
  meromorphic function on $L$.  Let $F = -\log\|s\|_\fL\big|_{\Gamma_\fX}$.
  Then $F$ is a tropical meromorphic function on $\Gamma_\fX$ and
  \begin{equation}\label{eq:section.of.trop.bundle}
    \tau_* \div(s) = \div(F) + \sum_{\zeta} \deg(\fL|_{D_\zeta})\, (\zeta),
  \end{equation}
  where the sum is taken over vertices $\zeta$ of $\Gamma_\fX$ and
  $D_\zeta$ is the irreducible component of $\fX_k$ with generic point
  $\red(\zeta)$. 
\end{theorem}

\begin{proof}
  If $e\subset\Gamma_\fX$ is an open edge, then $\red(\tau^{-1}(e))$ is a node in
  $\fX_k$, which is contained in a formal affine open subset of $\hat\fX$ on
  which $\fL$ is trivial.  Hence $F = -\log|f|$ on $A = \tau^{-1}(e)$ for some
  nonzero meromorphic function $f$ on $A$, so $F$ is piecewise affine-linear
  with integral slopes on $A$ and $\div(F|_A) = \tau_* \div(s|_A)$
  by~\cite[Proposition~2.10(1)]{baker_payne_rabinoff13:analytic_curves}.
  Since this holds for each edge, $F$ is a tropical meromorphic function on
  $\Gamma_\fX$.

  Now let $\zeta$ be a vertex of $\fX$, and let $D = D_\zeta$.  By blowing up
  $\fX$ we can add vertices to the interior of loop edges in $\Gamma_\fX$.
  Hence we may assume that $\Gamma_\fX$ has no loop edges, so that $D$ is
  smooth. After multiplying by a nonzero scalar we may also assume that
  $\|s(\zeta)\|_\fL = 1$, so that $s$ reduces to a nonzero meromorphic function
  $\td s$ on $D$.  Let $\td x\in D(k)$, and let $v_{\td x}$ be the tangent
  direction at $\zeta$ in the direction of $\red^{-1}(\td x)$
  (see \cite[(5.13)]{baker_payne_rabinoff13:analytic_curves}).  Let
  $\fU$ be an open neighborhood of $\td x$ trivializing $\fL$, and let
  $U = \red^{-1}(\fU_k)$, so $F = -\log|f|$ on $U$ for some nonzero meromorphic
  function $f$ on $U$ with a well-defined reduction $\td f$ on $\fU_k$.
  By~\cite[Theorem~5.15(3)]{baker_payne_rabinoff13:analytic_curves}%
  \footnote{This theorem is only stated for \emph{algebraic} meromorphic
    functions, but is true for analytic meromorphic functions such as $f$
    (see~\cite[Remark~3.6.6]{cohen_temkin_trushin14:differentials}).}
  we have
  \[ \ord_{\td x}(\td s) = \ord_{\td x}(\td f) = d_{v_{\td x}} F(\zeta). \]
  Combining this with~\cite[Proposition~2.10(2)]{baker_payne_rabinoff13:analytic_curves} yields
  \[ \ord_{\td x}(\td s) = \deg(\div(s|_{\red^{-1}(\td x)})) \]
  for all points $\td x\in D^\sm(k)$, the set of points of $D(k)$ which are not
  nodes in $\fX_k$.  Since the edges of $\Gamma_\fX$ adjacent to $\zeta$
  represent the tangent vectors at $\zeta$ in the direction of the points of
  $D(k)\setminus D^\sm(k)$, we combine the previous two equations to obtain
  \[\begin{split}
    \deg(\fL|_D) &= \sum_{\td x\in D(k)} \ord_{\td x}(\td s)
    = \sum_{\td x\in D^\sm(k)} \ord_{\td x}(\td s)
    + \sum_{\td x\in D(k)\setminus D^\sm(k)} \ord_{\td x}(\td s) \\
    &= \deg(\div(s|_{\tau^{-1}(\zeta)})) + 
    \sum_{\td x\in D(k)\setminus D^\sm(k)} d_{v_{\td x}} F(\zeta) \\
    &= \deg(\div(s|_{\tau^{-1}(\zeta)})) - \ord_\zeta(F).
  \end{split}\]
Equation~\eqref{eq:section.of.trop.bundle} follows.
\end{proof}

\begin{remark}
  As mentioned above, our slope formula is closely related to the
  Poincar\'e--Lelong formula in non-Archimedean Arakelov theory.  When the base
  is a discretely valued field, Theorem~\ref{thm:slope.formula.2} essentially
  goes back to Zhang~\cite{zhang93:admissible_pairing}.  The term
  $\sum_{\zeta} \deg(\fL|_{D_\zeta}) (\zeta)$
  in~\eqref{eq:section.of.trop.bundle} is precisely the measure $\hat c_1(\fL)$
  that Chambert-Loir~\cite{chambertloir06:mesures_equidist} associates to the
  formally metrized line bundle $L$, where $(\zeta)$ is interpreted as a point
  mass at $\zeta$.  In this language, we have
  \[ \hat c_1(\fL) = \div\big(\log\|s\|_\fL\big|_{\Gamma_\fX}\big) +
  \tau_*\div(s), \]
  where again the divisors are interpreted as counting measures.
  This is formally similar to the Poincar\'e--Lelong formula (see~\cite[Lemma~2.2.5]{chambertloir11:heights_measures} for a precise
  statement, still over a discretely valued base).
\end{remark}

\subsection{Integral Rosenlicht differentials}\label{sec:rosenlicht}
We will apply Theorem~\ref{thm:slope.formula.2} to sections of a certain
canonical extension $\Omega^1_{\fX/R}$ of the cotangent bundle $\Omega^1_{X/K}$
to our semistable model $\fX$.  If $R$ were discretely valued, we could define
$\Omega^1_{\fX/R}$ as the relative dualizing sheaf, or as the sheaf of
logarithmic differentials.  In the non-Noetherian case it is easiest to make a
somewhat ad-hoc construction, which we develop here as it is nonstandard.  To 
begin we may assume that $K$ is any complete non-Archimedean field with algebraically
closed residue field $k$.

\begin{definition}\label{def:rosenlicht}
  Let $\fX$ be a (not necessarily proper) semistable $R$-curve with smooth
  generic fiber, and let $j\colon\fU\injects\fX$ be the inclusion of the smooth
  locus.  The \defi{sheaf of integral Rosenlicht differentials} on $\fX$ is
  defined to be
  \[ \Omega^1_{\fX/R} \coloneq j_*\Omega^1_{\fU/R}, \]
  where $\Omega^1_{\fU/R}$ is the usual sheaf of K\"ahler differentials.
\end{definition}

\begin{example}\label{eg:basic.chart}
  Let $\fX = \Spec(R[S,T]/(ST-\varpi))$ for some $\varpi\in K^\times$ with
  $|\varpi| < 1$.  The smooth locus $\fU$ is the union of the two distinguished
  affine open subsets, where $S$ and $T$ are invertible.  Hence 
  \[ H^0(\fX,\Omega^1_{\fX/R}) = R[S^{\pm1}]\,\frac{dS}S\cap
  R[T^{\pm1}]\,\frac{dT}T \]
  inside of $K[S^{\pm 1}]\,dS/S = K[T^{\pm 1}]\,dT/T$.  Here we use that
  $S = \varpi/T$ and $dS/S = -dT/T$.  From this it is easy to see that
  $\Omega^1_{\fX/R}$ is a trivial invertible sheaf on $\fX$, with $dS/S=-dT/T$ a
  nonvanishing section.

  Note that a section $\omega\in H^0(\fX,\Omega^1_{\fX/R})$ restricts to a
  meromorphic section of the cotangent bundle on each component of the special
  fiber of $\fX$, with at worst a simple pole at the origin, and such that the
  residues at the origin at each component sum to zero.
\end{example}

\begin{lemma}\label{lem:rosenlicht.invertible}
  Let $\fX$ be a semistable $R$-curve as in
  Definition~\ref{def:rosenlicht}.
  \begin{enumerate}
  \item The sheaf $\Omega^1_{\fX/R}$ is invertible.
  \item If $f\colon\fX'\to\fX$ is an \'etale morphism of semistable $R$-curves, then
    $f^*\Omega^1_{\fX/R} = \Omega^1_{\fX'/R}$.
  \item The restriction of $\Omega^1_{\fX/R}$ to the special fiber $\fX_k$ is
    isomorphic to the relative dualizing sheaf of $\fX_k/k$.
  \end{enumerate}
\end{lemma}

\begin{proof}
  First we treat~(2).  Let $j'\colon\fU'\injects\fX'$ be the inclusion of the smooth
  locus of $\fX'$.  Then $f^{-1}(\fU) = \fU'$ and
  $f^*\Omega^1_{\fU/R} = \Omega^1_{\fU'/R}$, so
  $j'_*\Omega^1_{\fU'/R} = f^*j_*\Omega^1_{\fU/R}$ by cohomology and base change
  for flat morphisms.  The first assertion is an immediate consequence of this
  and Example~\ref{eg:basic.chart}, as every singular point of $\fX$ has an
  \'etale neighborhood which is \'etale over $\Spec(R[S,T]/(ST-\varpi))$
  for some $\varpi$.

  The Cartesian square
  \[\xymatrix @=.2in{
    {\fU_k} \ar[r]^i \ar[d]_{\bar j} & {\fU} \ar[d]^j  \\
    {\fX_k} \ar[r]_i & {\fX} }\]
  gives rise to a natural homomorphism
  $\phi\colon i^*\Omega^1_{\fX/R}\to\bar j_*\Omega^1_{\fU_k/k}$.  By construction this
  is an isomorphism on $\fU_k$.  Working \'etale-locally, it is clear from
  Example~\ref{eg:basic.chart} that $\phi$ is injective and that its image has
  the following description.  Let $\pi\colon\td\fX_k\to\fX_k$ be the normalization.
  Then a section in the image of $\phi$ in a neighborhood of a singular
  point $\td x\in\fX_k$ pulls back to a meromorphic section of
  $\Omega^1_{\td\fX_k/k}$ with at worst simple poles at the points of
  $\pi^{-1}(\td x)$, such that the residues sum to zero.  Therefore
  $i^*\Omega^1_{\fU_k/k}$ is the sheaf of classical Rosenlicht differentials,
  which is well known to be a dualizing sheaf.
\end{proof}

\subsubsection{Interpretation in terms of skeletons}
Now we suppose that $K$ is algebraically closed and that $\fX$ is a proper
semistable $R$-curve with smooth, connected generic fiber $X$.  As above, we let
$\Gamma_\fX$ denote the associated skeleton, considered as a vertex-weighted
metric graph.

\begin{lemma}\label{lem:rosenlicht.degree}
  Let $\zeta\in\Gamma_\fX$ be a vertex, and let
  $D_\zeta\subset\fX_k$ be the corresponding irreducible component.  Then
  \begin{equation}\label{eq:deg.rosenlicht}
    \deg(\Omega^1_{\fX/R}|_{D_\zeta}) = 2g(\zeta) - 2 + \deg(\zeta),
  \end{equation}
  where $g(\zeta)$ is the weight of $\zeta$ and $\deg(\zeta)$ is the
  valency of $\zeta$ in $\Gamma_\fX$.
\end{lemma}

\begin{proof}
  This is an immediate consequence of the definitions and
  Lemma~\ref{lem:rosenlicht.invertible}(3).
\end{proof}

The formal metric on $\Omega^1_{X/K}$ coming from $\Omega^1_{\fX/R}$ can be
computed explicitly on $\Gamma_\fX$, as follows.  Let $e\subset\Gamma_\fX$ be an
open edge, and let $A = \tau^{-1}(e)$ be an open annulus.  Choose an isomorphism
$T\colon A\isomto\bS(\varpi)_+$ with a standard open annulus.

\begin{lemma}\label{lem:nonv.sec.annulus}
  With the above notation, let $\omega = f(T)\,dT/T$ be the restriction of a
  nonzero meromorphic section of $\Omega^1_{X/K}$ to $A$.  Then
  $\|\omega\|_{\Omega^1_{\fX/R}}=|f|$ on $A$.
\end{lemma}

\begin{proof}
  First suppose that $T'\colon A\isomto\bS(\varpi)_+$ is a different isomorphism.  A
  calculation
  using~\cite[Proposition~2.2(1)]{baker_payne_rabinoff13:analytic_curves},
  the explicit description of the units on $\bS(\varpi)_+$, shows that
  $dT'/T' = g(T)\,dT/T$ for an invertible analytic function $g$ on $\bS(\varpi)_+$
  such that $|g(x)| = 1$ for all $x\in\bS(\varpi)_+$.  Hence the lemma is true
  for $T'$ if and only if it is true for $T$, so we may choose any parameter $T$
  that we like.

  Let $\td x\in\fX_k$ be the nodal point such that $A = \red^{-1}(\td x)$, let
  $\phi\colon\fU\to\fX$ be an \'etale neighborhood of $\td x$, and let $\td y\in\fU$
  be an inverse image of $\td x$.  Then $\phi$ induces an isomorphism
  $\red^{-1}(\td y)\isomto\red^{-1}(\td x) = A$.  Similarly, if
  $\psi\colon\fU\to\Spec(R[S,T]/(ST-\varpi))$ is an \'etale morphism sending $\td y$
  to the origin $\td z$, then $\psi$ induces an isomorphism
  $\red^{-1}(\td y)\isomto\red^{-1}(\td z)$.  Now the lemma follows from
  Example~\ref{eg:basic.chart}, where it was shown that $dT/T$ is a nonvanishing
  section of the sheaf of integral Rosenlicht differentials in a neighborhood of
  $\td z$.
\end{proof}

The next lemma says that the restriction of $\|\scdot\|_{\Omega^1_{\fX/R}}$ to
$\Gamma_\fX$ is compatible with refinement of the semistable model giving the skeleton.

\begin{lemma}\label{lem:indep.of.skel}
  Let $\fX,\fX'$ be two semistable models of $X$, and suppose that there exists
  a (necessarily unique) morphism $\fX'\to\fX$ inducing the identity on $X$.
  Then $\Gamma_\fX\subset\Gamma_{\fX'}$, and we have
  $\|\scdot\|_{\Omega^1_{\fX'/R}}\big|_{\Gamma_{\fX}} =
  \|\scdot\|_{\Omega^1_{\fX/R}}|_{\Gamma_{\fX}}$.
\end{lemma}

\begin{proof}
  The fact that $\Gamma_\fX\subset\Gamma_{\fX'}$ follows
  from~\cite[Proposition~3.13, Theorem~4.11]{baker_payne_rabinoff13:analytic_curves}.
  Then $\Gamma_{\fX'}$ is obtained from $\Gamma_\fX$ by subdividing some edges
  and adding some new ones.  As we are restricting to $\Gamma_\fX$, we are not
  concerned with new edges, so it suffices to show that
  $\|\scdot\|_{\Omega^1_{\fX/R}}$ is insensitive to subdividing an edge, or
  equivalently, to blowing up a node on $\fX$.  But by
  Lemma~\ref{lem:nonv.sec.annulus}, $\|\scdot\|_{\Omega^1_{\fX/R}}$ restricted
  to an open edge $e$ only depends on a parameter $T$ for $\tau^{-1}(e)$, and
  $T$ restricts to a parameter on $\tau^{-1}(e')$ for any $e'\subset e$.
\end{proof}

By virtue of Lemma~\ref{lem:indep.of.skel}, we will write
$\|\scdot\| = \|\scdot\|_{\Omega^1_{\fX/R}}\big|_{\Gamma_\fX}$ for any
semistable model $\fX$.

\begin{remark}
  Temkin~\cite{temkin:pluriforms} has developed an extremely general procedure
  for metrizing the cotangent sheaf on an analytic space, of which the above
  construction is a special case.  However, it is not obvious that the metric
  resulting from his theory restricts to a formal metric on a skeleton, and
  therefore one cannot immediately apply Theorem~\ref{thm:slope.formula.2}, as
  we do in Section~\ref{sec:canon.div.graph}.
\end{remark}

\subsubsection{Interpretation in terms of the canonical divisor of a graph}
\label{sec:canon.div.graph}
We assume still that $K$ is algebraically closed and that $\fX$ is a proper
semistable $R$-curve with smooth, connected generic fiber $X$.  The
\defi{canonical divisor} on $\Gamma_\fX$ is by definition
\begin{equation}\label{eq:canon.div.graph}
  K_{\Gamma_\fX} \coloneq \sum \big(2g(\zeta) - 2 + \deg(\zeta)\big)\,(\zeta),
\end{equation}
where the sum is taken over the vertices of $\Gamma_\fX$ (see~\cite[Definition~2.13]{abbr14:lifting_harmonic_morphism_I}).
By Lemma~\ref{lem:rosenlicht.degree}, equation~\eqref{eq:section.of.trop.bundle} becomes
\begin{equation}\label{eq:cont.specialization}
  \tau_*\div(\omega) = \div(F) + K_{\Gamma_\fX},
\end{equation}
where $\omega$ is a nonzero meromorphic $1$-form on $X$, and
$F = -\log\|\omega\|$.  In particular, if $\omega$ is a \emph{regular}
global $1$-form, then 
\begin{equation}\label{eq:section.of.trop.can.bun}
  \div(F) + K_{\Gamma_\fX}\geq 0,
\end{equation}
which formally says that $F$ is a ``section of the tropical canonical bundle.''

\subsubsection{Interpretation in terms of intersection theory}

Assume for this subsection that $K$ is discretely valued and $\fX$ is 
semistable, with irreducible decomposition $\fX_k = \bigcup\, C_i$ and with dual
graph $\Gamma_{\fX}$.
Let $\fL \in \Pic(\fX)$ be a line bundle, and denote by $\fL|_{\Gamma_{\fX}}$ the
divisor $\sum \left(\deg \fL|_{C_i}\right) C_i \in \Div(\Gamma_{\fX})$.  Here we
are identifying irreducible components of $\fX_k$ with vertices of $\Gamma_\fX$. A nonzero
regular section $s$ of $\fL|_X$ extends to a meromorphic section of $\fL$, and
after scaling by an element of $K$ extends to a regular section of $\fL$ (with
possible zeroes along entire components of $\fX_k$).

Write $\divv(s) = H + V$, where $H$ is the closure of $\divv(s)|_X$ and
$V = \sum n_iC_i$ is the complement $\divv(s) - H$ (so that
$\Supp(V) \subset \fX_k$). Suppose that the support of $H$ is contained in
$\fX^{\text{reg}}$; this is guaranteed if $\fX$ is regular and
$\Supp (\divv(s))|_X \subset X(K)$. Then $\sO(H) \in \Pic(\fX)$;
if additionally $\fX$ is regular, then
$\sO(H)|_{\Gamma_{\fX}} = \sum \deg (C_i \cap H)\,C_i$.

Let $f\colon \Gamma_{\fX} \to \Z$ be given by $f(C_i) = n_i$ and extended linearly on edges of $\Gamma_{\fX}$. Then, since $\fX$
is regular, the adjunction formula \cite[Theorem~9.1.36]{liu:algebraicGeometry}
gives 
\[
\sO(V)|_{\Gamma_{\fX}} = -\Delta(f) \coloneq \sum_v \sum_{e=vw} (f(w) - f(v)) (v).
\]
Since
$\fL \cong \sO(\divv(s)) \cong \sO(V) \otimes \sO(H)$, this gives the geometric
variant of Baker's specialization lemma \cite[Lemma~2.8]{Baker:specialization}:
\[
\fL|_{\Gamma_{\fX}} + \Delta(f) =  \sO(H)|_{\Gamma_{\fX}}.
\]
When $\fL$ is the relative dualizing sheaf $\omega_{\fX/R}$, this is precisely
the ``discrete'' version of~\eqref{eq:cont.specialization}, as
$\omega_{\fX/R}|_{\Gamma_\fX} = K_{\Gamma_\fX}$, and $\Delta(f)$ (resp., 
$\sO(H)_{\Gamma_\fX}$) plays the role of $\div(F)$ (resp., $\tau_*\div(s)$).
From the point of view of chip firing, this formula has a more colloquial description: the vanishing of $s$ along components of $\fX_k$ gives exactly the firing sequence witnessing the linear equivalence of $\fL|_{\Gamma_{\fX}}$ with the divisor $\sO(H)|_{\Gamma_{\fX}}$ on the graph $\Gamma_{\fX}$.

\subsection{Basic wide open subdomains}\label{sec:basic.wide.open}
Assuming now that $K$ is algebraically closed, we return to the notation
of Section~\ref{sec:skeletons}.  Fix a split semistable $R$-model $\fX$ of $X$ and a
type-$2$ point $\zeta\in\Gamma_\fX$.
An \defi{open star neighborhood} of $\zeta$ in $\Gamma_\fX$ is a
simply connected open neighborhood of the form
$V = \{\zeta\}\djunion\bigcup_{i=1}^r e_i$, where $e_i$ is an open interval of
length in $\Lambda$ contained in an edge of $\Gamma_\fX$ and containing
$\zeta$ in its closure, and $r\geq 1$.  The inverse image $U = \tau^{-1}(V)$ of
an open star neighborhood is called a \defi{basic wide open subdomain} of
$X^\an$.  The \defi{central point} of $U$ is $\zeta$ and the
\defi{underlying affinoid} of $U$ is $Y \coloneq \tau^{-1}(\zeta)$.
After a suitable blowing up on the special fiber, we can arrange that $\zeta$ is
a vertex of $\Gamma_\fX$.  In this case, if
$D_\zeta\subset\fX_k$ is the irreducible component with generic point
$\td\zeta =\red(\zeta)$, then $Y = \red^{-1}(D_\zeta^\sm)$, where $D_\zeta^\sm$
is the set of smooth points of $\fX_k$ lying on $D_\zeta$.  Hence
\[ Y\setminus\{\zeta\}\cong\Djunion_{\td x\in D_\zeta^\sm}\red^{-1}(\td x), \]
where for any smooth point $\td x\in\fX_k$, the inverse image
$\red^{-1}(\td x)$ is isomorphic to an open unit disc.  Moreover, we have
$U\setminus Y = \Djunion\tau^{-1}(e_i)$, with each
$A_i \coloneq\tau^{-1}(e_i)$ isomorphic to an open annulus.  The closure of
$A_i$ in $X^\an$ is equal to $A_i\djunion\{\zeta_i,\zeta\}$, where
$\zeta_i\in\Gamma_\fX$ is the other endpoint of $e_i$, which is a type-$2$
point not contained in $U$.  We call $\zeta_i$ the \defi{end} of $U$ associated
to $A_i$. 

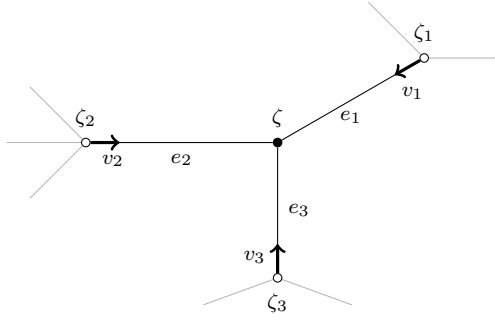
\begin{figure}[ht]
  \centering
  \begin{tikzpicture}[scale=1.5,
    every circle node/.style={draw, inner sep=.4mm},
    every node/.style={font=\tiny},
    other edge/.style={black!30!white, very thin}]
    \node (Z) at (0,0) [circle, fill=black, label=above:$\zeta$] {};
    \draw (Z) ++(30:1.5cm) 
          node (Z1) [circle, label=above:$\zeta_1$] {};
    \draw[other edge] (Z1) -- ++(0:.7cm)
          (Z1) -- ++(135:.7cm);
    \draw (Z) ++(180:1.7cm) 
          node (Z2) [circle, label=above:$\zeta_2$] {};
    \draw[other edge] (Z2) -- ++(135:.7cm)
          (Z2) -- ++(180:.7cm)
          (Z2) -- ++(225:.7cm);
    \draw (Z) ++(270:1.2cm) 
          node (Z3) [circle, label=below:$\zeta_3$] {};
    \draw[other edge] (Z3) -- ++(200:.7cm)
          (Z3) -- ++(340:.7cm);
    \path (Z) edge node [below] {$e_1$} (Z1) 
          (Z) edge node [below] {$e_2$} (Z2) 
          (Z) edge node [right] {$e_3$} (Z3);
    \draw[very thick, ->] (Z1) -- ($(Z1)!.3cm!(Z)$)
          node [anchor=135] {$v_1$};
    \draw[very thick, ->] (Z2) -- ($(Z2)!.3cm!(Z)$)
          node [anchor=70] {$v_2$};
    \draw[very thick, ->] (Z3) -- ($(Z3)!.3cm!(Z)$)
          node [anchor=30] {$v_3$};
  \end{tikzpicture}
  
  \caption{An open star neighborhood of a type-$2$ point $\zeta$ in a skeleton
    $\Gamma_\fX$ and associated notations.  The tangent vectors illustrate the
    statement of Proposition~\ref{prop:zeros.on.wide.open}.}
  \label{fig:wide.open}
\end{figure}

\begin{remark}
  Our definitions of basic wide open subdomains and their ends are equivalent to
  those of Coleman~\cite[Section~3]{ColemanReciprocity} under the identification of a
  Berkovich analytic space and its corresponding rigid space.  More precisely,
  any basic wide open subdomain of $X^\an$ in Coleman's sense is the inverse
  image of an open star neighborhood of a type-$2$ point of \emph{some} skeleton
  $\Gamma$ of $X$.
\end{remark}

\begin{remark}
  The open star neighborhood $V$ deformation retracts onto $\zeta$, and the
  deformation retraction of $X^\an$ onto its skeleton retracts $U$ onto $V$.
  Therefore a basic wide open subdomain is contractible.
\end{remark}

We will use the following fundamental relationship between the number of zeros
of an analytic function on $U$ and the slopes of its valuation at the ends.
This is called the ``mass formula''
in~\cite[Proposition~5.30]{baker_rumely10:book}, where it is proved for
basic wide open subdomains of $\PP^{1,\an}$. The situation in the proposition below is illustrated in Figure~\ref{fig:wide.open}.

\begin{proposition}\label{prop:zeros.on.wide.open}
  Let $U\subset X^\an$ be a basic wide open neighborhood with underlying
  affinoid $Y$, annuli $A_1\djunion\cdots\djunion A_r = U\setminus Y$, and
  corresponding ends $\zeta_1,\ldots,\zeta_r$.  Let $v_i$ denote the tangent
  direction at $\zeta_i$ in the direction of $A_i$.  Let $f$ be a nonzero
  meromorphic function on $U$ which extends to a meromorphic function on a
  neighborhood of $\{\zeta_1,\ldots,\zeta_r\}$, and let $F = -\log|f|$.  Then
  \[ \deg\big(\div(f|_U)\big) = \sum_{i=1}^r d_{v_i} F(\zeta_i). \]
\end{proposition}

\begin{proof}
  This follows from Theorem~\ref{thm:slope.formula} and an easy combinatorial
  argument. 
\end{proof}
\subsection{de Rham cohomology of a basic wide open}
\label{sec:de.rham}
In this section we assume that $K = \C_p$ for a prime $p$.  We will need Coleman's calculation of the
de Rham cohomology of a basic wide open $U = Y\djunion\Djunion_{i=1}^r A_i$.
This calculation does not depend on the ambient curve $X$, so by gluing closed
discs onto the annuli $A_i$ we may assume that $X$ has good reduction and that
$U$ is the complement in $X$ of finitely many closed discs contained in distinct
residue discs.  Let $S\subset X(\C_p)\setminus U(\C_p)$ be a choice of $r$
points, one in each deleted disc.

For a scheme $Z$ over $\C_p$, we let $H^1_\dR(Z)^\alg$ denote the algebraic de
Rham cohomology of $Z$ over $\C_p$, so $H^1_\dR(X)^\alg$ is a $\C_p$-vector
space of dimension $2g$, where $g$ is the genus of $X$.  We define
\[ H^1_\dR(U) = \Omega^1_{U/\C_p}(U)/d\sO(U), \]
the analytic differential forms modulo the exact differentials.
Coleman~\cite[Theorem~4.2]{ColemanReciprocity} proves that the natural
restriction map
\[ H^1_\dR(X\setminus S)^\alg \To H^1_\dR(U) \]
is an isomorphism.  In particular,
\begin{equation}\label{eq:dim.h1}
\dim_{\C_p} H^1_\dR(U) = \dim_{\C_p} H^1_\dR(X\setminus S)^\alg = 2g - 1 +
\#S = 2g - 1 + r,
\end{equation}
where $r$ is the number of deleted discs.  (The algebraic de Rham cohomology can
be calculated using a comparison theorem over $\C$, for example.)

Let $T$ be a parameter on the annulus $A_i$, normalized so that $|T(x)|\searrow 1$
as $x\to\zeta$, the central point of $U$.  Let
$\omega\in\Omega_{A_i/\C_p}^1(A_i)$, and write
\[ \omega = \sum_{n=-\infty}^\infty a_n T^n \frac{dT}T. \]
The \defi{residue} of $\omega$ is defined to be $\Res(\omega) = a_0$.  This is
independent of the parameter $T$ up to a sign that is determined by the
orientation of the annulus, which we have fixed.  Clearly
the residue of an exact differential is zero, so $\Res$ defines a homomorphism
$H^1_\dR(A_i)\to\C_p$.

\begin{theorem}[Coleman~{\cite[Proposition~4.3, Proposition~4.4]{ColemanReciprocity}}]\label{thm:H1dR}
  The following sequence is exact:
  \[ 0 \To H^1_\dR(X)^\alg \To H^1_\dR(U) \xrightarrow{\Dsum\Res}
  \Dsum_{i=1}^r \C_p \overset{\sum}\To \C_p \To 0. \]
\end{theorem}

\begin{proof}
  By~\cite[Proposition~4.4]{ColemanReciprocity}, the sequence is exact at
  $H^1_\dR(U)$, so we only need to justify exactness at $\Dsum_{i=1}^r\C_p$.
  Proposition~4.3 in~\cite{ColemanReciprocity} says that the image is contained in the
  kernel, so exactness follows from the dimension count~\eqref{eq:dim.h1}.
\end{proof}

\section{Integration}
\label{sec:integration}

We will use two integration theories on curves, namely, Berkovich--Coleman
integration and abelian integration.  The former is functorial with respect to
morphisms and can be calculated by formal antidifferentiation on open annuli.
The latter is suitable for use with Chabauty's method.  The purpose of this
section is to introduce the two integrals and compare them.  Related work comparing the two integrals in the context of parallel transport is being undertaken by Besser and Zerbes \cite{besserzerbes}.

In this section, we take $K = \C_p$, with the valuation normalized so that
$\val(p)=1$.  We introduce the following notation for a smooth, commutative
algebraic or analytic $\C_p$-group $G$ and a smooth $\C_p$-analytic space $X$:
\smallskip

\begin{tabular}{ll}
  $Z^1_\dR(X)$ & The space of closed $1$-forms on $X$.  \\
  $\Omega^1_\inv(G)$ & $\subset Z^1_\dR(G)$, the space of invariant $1$-forms on
  $G$. \\
  $\Lie(G)$ & The tangent space of $G$ at the identity, dual to
  $\Omega^1_\inv(G)$. \\
\end{tabular}

\subsection{Integration theories}\label{s:twointegrals}

Let $X$ be a smooth $\C_p$-analytic space, and let $\calP(X)$ be the set of paths
$\gamma\colon [0,1]\rightarrow X$ with ends in $X(\C_p)$.

\begin{definition} \label{d:it} An \defi{integration theory} on $X$ is a map $\int\colon \calP(X)\times Z^1_\dR(X)\rightarrow\C_p$ such that:
\begin{enumerate}
\item \label{i:ftc} If $U\subset X$ is an open subdomain isomorphic to an open
  polydisc and $\omega|_U=df$ with $f$ analytic on $U$, then
  $\int_\gamma \omega=f(\gamma(1))-f(\gamma(0))$ for all $\gamma\colon [0,1]\to U$.

\item $\int_\gamma \omega$ only depends on the fixed endpoint homotopy class of $\gamma$.

\item If $\gamma'\in\calP(X)$ and $\gamma'(0)=\gamma(1)$, then
\[\int_{\gamma'*\gamma} \omega=\int_\gamma \omega+\int_{\gamma'} \omega. \]

\item $\omega\mapsto\int_\gamma \omega$ is linear in $\omega$ for fixed
  $\gamma$.
\end{enumerate}
\end{definition}

Condition (\ref{i:ftc}) completely determines the integration theory on an open
polydisc $X$ by the Poincar\'e lemma: every closed $1$-form
$\omega\in Z^1_\dR(X)$ is exact.  To be explicit, let $X = \bB(1)_+$ be the
$1$-dimensional open unit disc.  Any $\omega\in Z^1_\dR(\bB(1)_+)$ can be
written as $\omega = g(T) dT$, where $g(T) = \sum_{n\geq 0} a_n T^n$ is a
convergent power series; then $\omega = df$, where
$f(T) = \sum_{n\geq 0} a_n T^{n+1} / (n+1)$ is the power series obtained by
formally antidifferentiating $g(T)$.  Note that on an \emph{open} disc, if
$g(T)$ is convergent, then $f(T)$ is also convergent.  Hence for
$\gamma\in\calP(X)$ we have $\int_\gamma g(T)dT = f(\gamma(1)) - f(\gamma(0))$.
In higher dimensions one proceeds as above, one variable at a time, as in the
proof of the classical Poincar\'e lemma.

In general, Definition~\ref{d:it} does not completely specify an integration
theory, because a smooth $\C_p$-analytic space, even a smooth proper curve,
cannot necessarily be covered (as a Berkovich space) by open polydiscs.  The
ambiguity is illustrated in the following fundamental example.

\begin{example} \label{eg:torus.integral}
  Let $\bG_m^\an$ be the analytification of the multiplicative group over $\C_p$
  with coordinate $T$.  This is a contractible topological space, so any
  integration theory on $\bG_m^\an$ is by definition path-independent; hence, it
  makes sense to write $\int_x^y\omega$ for $x,y\in\C_p^\times$ and
  $\omega\in Z^1_\dR(\bG_m^\an) = H^0(\bG_m^\an, \Omega^1_{\bG_m^\an/\C_p})$. 

  Let $\omega=dT/T$, an invariant $1$-form.  The formal
  antiderivative of $\omega$ on the space of ``1-units''
  $U_1 = \{x\in\bG_m^\an~:~|T(x)-1| < 1\}\cong\bB(1)_+$ is the logarithm given
  by the usual Mercator series
  \[ \log(T) = \sum_{n=1}^\infty(-1)^{n+1}\frac{(T-1)^n}n. \]
  If we require that $x\mapsto\int_1^x dT/T \colon \C_p^\times\to\C_p$ be a group
  homomorphism, then $\int_1^x dT/T$ is determined on
  $\sO_{\C_p}^\times = \{x\in\C_p^\times~:~|x|=1\}$ by the property that
  $\int_1^x dT/T = 0$ for $x$ a root of unity.  Set $\log(x) = \int_1^x dT/T$
  for $x\in\sO_{\C_p}^\times$.

  Beyond this, one has to make a choice to specify an integration theory on
  $\bG_m^\an$.  Let $t\colon \Q\to\C_p^\times$ be a section of $\val\colon \C_p^\times\to\Q$
  such that $t(1) = p$, and let $h\colon \Q\to\C_p$ be any additive group
  homomorphism.  Define $\int_1^x dT/T = \log(x_1) + h(r)$, where
  $x = x_1\cdot t(r)$ for $r\in\Q$ and $x_1\in\sO_{\C_p}^\times$.  It turns out that this
  extends to an integration theory on $\bG_m^\an$ for any choice of $h$.  (The
  definition does not depend on $t$, since if $t'$ is another section with
  $t(1)=p$, then $t(r)t'(r)^{-1}$ is a root of unity for all $r\in\Q$.)

\end{example}

\subsection{Berkovich--Coleman integration}\label{sec:bc.integral}
The choice of homomorphism $h\colon \Q\to\C_p$ in Example~\ref{eg:torus.integral} is
the only additional datum necessary for the construction of the
Berkovich--Coleman integration theory.  It is equivalent to the following datum.

\begin{definition} A \defi{branch of the logarithm} is a homomorphism
\[ \Log \colon
 \C_p^\times \To \C_p \]
that restricts to $\log$ on $\sO_{\C_p}^\times$.
\end{definition}

In the notation of Example~\ref{eg:torus.integral}, we have
\begin{equation}\label{eq:Log.and.h}
  \Log(x) = \log(x_1) + h(r), \sptxt{where} x = x_1\cdot t(r).
\end{equation}
After mandating that the integral be functorial under morphisms of analytic
spaces, the integration theory is uniquely specified by equivariance under a
lift of Frobenius, a principle attributed to Dwork.  This approach to
integration has been greatly extended by
Berkovich~\cite{Berkovich:integrationOfOne}; here, we present only a very small
subset of his theory.

\begin{definition} \label{def:bc.int}
  The \defi{Berkovich--Coleman integration theory} is an integration theory
  \[ \presuper\BC\int\colon \calP(X)\times Z^1_\dR(X)\To\C_p \]
  for every smooth $\C_p$-analytic space $X$, satisfying:
  \begin{enumerate}
  \item if $X = \G_m^{\an}$, then $\presuper\BC\int_1^x dT/T=\Log(x)$, and
  \item if $f\colon X\rightarrow Y$ is a morphism and $\omega\in Z^1_\dR(Y)$, then
    \[\presuper\BC\int_\gamma f^*\omega=\presuper\BC\int_{f(\gamma)} \omega.\]
  \end{enumerate}
  Moreover, condition~(\ref{i:ftc}) of Definition~\ref{d:it} holds for any
  open subdomain $U\subset X$.
\end{definition}

This integration theory was defined for curves of bad reduction by Coleman and de Shalit \cite{CdS}.  There, one covers a curve by basic wide open subsets and
annuli.  A primitive (i.e., an antiderivative) is produced on the basic wide
opens by means of Frobenius equivariance and constructed explicitly on annuli by
antidifferentiating a power series.  This integration theory is closely related to
that of Schneider on $p$-adically uniformized curves (see
\cite{deShalit:versus} for details on the comparison).

\begin{example}
  Choose $\varpi\in\C_p$ with $0<|\varpi|<1$, and let $X = \bS(\varpi)_+$, the
  open annulus $|\varpi| < |x| < 1$.  A closed $1$-form $\omega$ can be written
  \[ \omega = g(T)\frac{dT}T = \sum_{n=-\infty}^\infty a_n T^n\,\frac{dT}T \]
  for a convergent infinite-tailed Laurent series $g(T)$.  Let
  $f(T) = \sum_{n\neq 0} (a_n/n)T^n$.  Then $df = \omega - a_0(dT/T)$, so for
  $x,y\in\bS(\varpi)(\C_p)$, we have
  \[ \presuper\BC\int_x^y \omega = \big(f(y) + a_0\Log(y)\big) - \big(f(x) + a_0\Log(x)\big). \]
\end{example}

\begin{example}\label{ex:BC.homom}
  Let $G$ be a smooth, commutative, simply connected $\C_p$-analytic group, and
  let $\omega$ be a (closed) invariant differential on $G$.  Since $G$ is
  simply connected, a Berkovich--Coleman integral only depends on the endpoints
  of a path, so it makes sense to write $\presuper\BC\int_1^x\omega$ for
  $x\in G(\C_p)$.  For $x,y\in G(\C_p)$, we have
  \[ \presuper\BC\int_1^x\omega + \presuper\BC\int_1^y\omega =
  \presuper\BC\int_1^x\omega + \presuper\BC\int_x^{xy}\omega =
  \presuper\BC\int_1^{xy}\omega, \]
  where the first equality is by invariance of $\omega$ and the second is by
  Definition~\ref{d:it}(3). Therefore $x\mapsto\presuper\BC\int_1^x\omega$ is a
  group homomorphism $G(\C_p)\to\C_p$.
\end{example}

In what follows, we will pick once and for all a branch of logarithm.
A convenient choice is $\Log(p)=0$; that is, $h=0$.

\subsection{The abelian integral}
\label{sec:abelian-integral}
Another approach to defining a $p$-adic integration theory on
a curve is via $p$-adic Lie theory on its Jacobian.  This was done in great
generality by Zarhin~\cite{Zarhin:abelian-integrals}.  This method was
extended to the $p$-adic Tate module of abelian varieties by
Colmez~\cite{Colmez:periodes}.
Other references for this integration theory are
\cite{Breuil:integration}, \cite{Vologodsky:hodge} and the second part of
\cite{ColemanIovita}, taken with the understanding that the first part uses the
Berkovich--Coleman integration theory.

Recall that if $A$ is an abelian variety over $K$, then
\[ \Omega^1_{A/K}(A) = \Omega^1_\inv(A) = Z^1_\dR(A) \]
because all global $1$-forms on a proper group scheme are invariant, and any
invariant $1$-form on a smooth, commutative algebraic group is closed.

\begin{definition}\label{def:abelian.integral}
  Let $A$ be an abelian variety over $\C_p$.  The \defi{abelian logarithm} on $A$
  is the unique homomorphism of $\C_p$-Lie groups
  $\log_{A(\C_p)}\colon  A(\C_p)\to\Lie(A)$ such that
  \begin{equation}\label{eq:tangent.cond}
    d\log_{A(\C_p)} \colon
 \Lie(A) \To \Lie(\Lie(A)) = \Lie(A)
  \end{equation}
  is the identity map.
\end{definition}

See~\cite{Zarhin:abelian-integrals} for the existence and uniqueness of
$\log_{A(\C_p)}$.  For $x\in A(\C_p)$ and $\omega\in\Omega_{A/\C_p}^1(A)$, we define
\[ \presuper\Ab\int_0^x\omega = \angles{\log_{A(\C_p)}(x),\,\omega}, \]
where $\angles{\scdot,\scdot}$ is the pairing between $\Lie(A)$ and
$\Omega_{A/\C_p}^1(A)$.  For $x,y\in A(\C_p)$, we set
\[ \presuper\Ab\int_y^x\omega = \presuper\Ab\int_0^x\omega -
\presuper\Ab\int_0^y\omega. \]
We call $\presuper\Ab\int$ the \defi{abelian integral} on $A$.

The abelian logarithm and the abelian integral are functorial under homomorphisms
of abelian varieties: if $f\colon A\rightarrow B$ is a homomorphism, then
$df\circ \log_{A(\C_p)} = \log_{B(\C_p)}\circ f$ and
\[\presuper{\Ab}{\int}_\gamma f^*\omega=\presuper{\Ab}{\int}_{f(\gamma)}
\omega\]
for $\omega\in\Omega_{B/\C_p}^1(B)$.

\begin{proposition}\label{prop:abelian.int}
  The abelian integral
  $(\gamma,\omega)\mapsto\presuper\Ab\int_{\gamma(0)}^{\gamma(1)}\omega$ is an
  integration theory on $A^\an$ in the sense of Definition~\ref{d:it}.
\end{proposition}

We postpone the proof until after the comparison result,
which is Proposition~\ref{prop:compare.integrals}.

\subsection{Comparison between the Berkovich--Coleman and abelian integrals}
\label{sec:comparison.thm}

Before comparing the two integration theories, we consider the following
motivating example (see also~\cite[Ex.~7.4]{ColemanIovita}).

\begin{example}
  Let $E$ be an elliptic curve over $\C_p$ with bad reduction.  Then $E$ is a
  Tate curve; that is, it has a uniformization $E^\an\cong\bG_m^\an/q^\Z$ for a
  unique value $q\in\C_p$ with $0<|q|<1$.  As $\bG_m^\an$ is contractible, the
  projection $\pi\colon \bG_m^\an\to E^\an$ is a universal covering space (in
  the sense of point-set topology), with deck transformation group $q^\Z$.

  Let $\omega$ be the invariant $1$-form on $E$ which pulls back to $dT/T$ on
  $\bG_m^\an$.  Let $\gamma\colon [0,1]\to E^\an$ be a path with $\gamma(0) = 0$ and
  $\gamma(1) = x\in E(\C_p)$, let $\td\gamma\colon [0,1]\to\bG_m^\an$ be the unique lift
  of $\gamma$ with $\td\gamma(0) = 1$, and let $\td x = \td\gamma(1)$.  By
  the definition and the functoriality of the Berkovich--Coleman integral, we have
  \[ \presuper\BC\int_{\gamma}\omega = \presuper\BC\int_{\td\gamma}\frac{dT}T =
  \Log(\td x). \]
  On the other hand, the abelian integral gives rise to a (potentially)
  \emph{different} branch of the logarithm $\Log_{\Ab}\colon \C_p^\times\to\C_p$ by
  setting
  \[ \Log_{\Ab}(\td x) \coloneq \presuper\Ab\int_0^{\pi(\td x)}\omega =
  \presuper\Ab\int_\gamma\omega. \]
  This branch of the logarithm $\Log_{\Ab}$ comes from the homomorphism
  $h_\Ab\colon \Q\to\C_p$ defined by $\Q$-linearity and
  $h_\Ab(\val(q)) = \log(t(\val(q))/q)$.  As both $\Log$ and $\Log_\Ab$ restrict
  to $\log$ on $\sO_{\C_p}^\times$, their difference $\Log-\Log_\Ab$ descends to
  a homomorphism from $\C_p^\times/\sO_{\C_p}^\times = \Q$ to $\C_p$, and we
  have
  \[ \presuper\BC\int_\gamma\omega - \presuper\Ab\int_\gamma\omega = \Log(\td x)
  - \Log_\Ab(\td x) = (h - h_\Ab)(\val(\td x)). \]
  In particular, the difference between the integrals is a $\Q$-linear function
  of the valuation of $\td x$.  We will show that this fact, suitably
  interpreted, holds in general.
\end{example}

\subsubsection{Non-Archimedean uniformization of abelian varieties}\label{sec:uniformization}
To study the general situation we will make use of the non-Archimedean analytic
uniformization of abelian varieties, in Berkovich's language.  The canonical
references
are~\cite{bosch_lutkeboh84:stable_reduction_uniformiz_abelian_varieties_II} and ~\cite{bosch_lutkeboh91:degenerat_abelian_varieties} (see also~\cite[Section~4]{baker_rabinoff} for a summary).

Let $A$ be an abelian variety over $\C_p$, and let $\pi\colon  E^\an\to A^\an$ be the
topological universal cover of $A^\an$.  Then $E^\an$ has the unique structure
of a $\C_p$-analytic group (after choosing an identity element), and the kernel
$M'$ of $\pi$ is canonically isomorphic to $\pi_1(A^\an) = H_1(A^\an,\Z)$.
Moreover, $E^\an$ is the analytification of an algebraic $\C_p$-group $E$, which
is an extension of an abelian variety $B$ with good reduction by a torus $T$.
This uniformization theory is summarized in the
\defi{Raynaud uniformization cross}:
\begin{equation} \label{eq:raynaud.cross}
  \xymatrix @=.20in{
    & {T^\an} \ar[d] & \\
    {M'} \ar[r] & {E^\an} \ar[r]^\pi\ar[d] & {A^\an} \\
    & {B^\an} &
  }
\end{equation}
Let $M$ be the character lattice of $T$, so $T = \Spec(\C_p[M])$.
The abelian variety $A$ has a semiabelian $\sO_{\C_p}$-model $\sA$ whose
special fiber $\bar A$ fits into the short exact sequence
\[ 0 \To \bar T \To \bar A \To \bar B \To 0, \]
where $\bar T = \Spec(\bar\F_p[M])$ and $\bar B$ is the reduction of $B$.  Let
$\hat\sA$ be the $p$-adic completion of $\sA$, and let $A_0 = \hat\sA_\eta$ be
its analytic generic fiber.  This is an analytic domain subgroup in $A^\an$
which lifts in a unique way to an analytic domain subgroup in $E^\an$.  It fits
into the short exact sequence
\begin{equation} \label{eq:T0.A0.B}
  0 \To T_0 \To A_0 \To B^\an \To 0,
\end{equation}
where $T_0 = \sM(\C_p\angles{M})$ is the affinoid torus inside of $T^\an$.

Let $N = \Hom(M,\Z)$ be the dual of the character lattice of $T$, with
$(\scdot,\scdot)\colon  N\times M\to\Z$ the evaluation pairing.  For $u\in M$, we let
$\chi^u\in \C_p[M]$ denote the corresponding character of $T$.  We have a
\defi{tropicalization map} $\trop\colon  T^\an\to N_\R = \Hom(M,\R)$ defined by
$(\trop(\|\scdot\|),\,u) = -\log\|\chi^u\|$, where we regard $T^\an$ as a space
of seminorms on $\C_p[M]$.  The map $\trop$ is surjective, continuous, and proper,
and the affinoid torus $T_0$ is equal to $\trop^{-1}(0)$.  We extend $\trop$ to
all of $E^\an$ by declaring that $\trop^{-1}(0) = A_0$.  We have
$\trop(E(\C_p)) = N_\Q = \Hom(M,\Q)$, so the map $\trop\colon  E(\C_p)\to N_\Q$ is
a surjective group homomorphism with kernel $A_0(\C_p)$.  The restriction of
$\trop$ to $M'\subset E(\C_p)$ is injective, and its image
$\trop(M')\subset N_\Q$ is a full-rank lattice in the real vector space
$N_\R$.  Let $\Sigma = \Sigma(A)$ be the real torus $N_\R/\trop(M)$.  Since
$A^\an$ is the quotient of $E^\an$ by the action of $M'$, there exists a unique
map $\bar\tau\colon  A^\an\to\Sigma$ making the following exact diagram commute:
\begin{equation} \label{eq:trop.of.unif}
  \xymatrix @R=.2in{
    & & A_0 \ar[d] \ar @{=} [r] & A_0 \ar[d] & \\
    0 \ar[r] & {M'} \ar[r] \ar[d]^{\trop}_\cong
    & {E^\an} \ar[r] \ar[d]^{\trop}
    & {A^\an} \ar[r] \ar[d]^{\bar\tau} & 0 \\
    0 \ar[r] & {\trop(M')} \ar[r] 
    & {N_\R} \ar[r] & {\Sigma} \ar[r] & 0
  }
\end{equation}
The real torus $\Sigma$ is called the \defi{skeleton} of $A$; in fact there
exists a canonical section $\Sigma\to A^\an$ of $\bar\tau$, and $A^\an$
deformation retracts onto its image~\cite[Section~6.5]{berkovic90:analytic_geometry}.
Letting $\Sigma_\Q = N_\Q/\trop(M')$ and taking $\C_p$-points, we
have a surjective homomorphism of short exact sequences
\begin{equation} \label{eq:trop.of.unif.pts}
  \xymatrix @R=.2in{
    0 \ar[r] & {M'} \ar[r] \ar[d]^{\trop}_\cong
    & {E(\C_p)} \ar[r] \ar[d]^{\trop}
    & {A(\C_p)} \ar[r] \ar[d]^{\bar\tau} & 0 \\
    0 \ar[r] & {\trop(M')} \ar[r] 
    & {N_\Q} \ar[r] & {\Sigma_\Q} \ar[r] & 0
  }
\end{equation}
where $A_0(\C_p)$ is the kernel of the middle and right vertical arrows.

\subsubsection{Comparison of the integrals}
Since $E^\an$ is locally isomorphic to $A^\an$, or since any invariant $1$-form
on $E^\an$ descends to an invariant $1$-form on $A^\an$, we have canonical
identifications 
\[ \Lie(E) = \Lie(A) \sptxt{and} \Omega^1_\inv(E) = \Omega_{A/\C_p}^1(A). \]
Since $E^\an$ is simply connected, by Example~\ref{ex:BC.homom} we may
define a homomorphism
\[ \log_\BC \colon
 E(\C_p)\To\Lie(A), \qquad x\mapsto\presuper\BC\int_0^x. \]
Composing the abelian logarithm $\log_{A(\C_p)}\colon A(\C_p)\to\Lie(A)$ with
$\pi\colon E^\an\to A^\an$ yields
\[ \log_\Ab \colon
 E(\C_p)\To\Lie(A), \qquad x\mapsto\presuper\Ab\int_0^{\pi(x)}. \]
\begin{proposition}\label{prop:compare.integrals}
  The difference 
  \[ \log_{\BC}-\log_{\Ab} \colon
 E(\C_p)\To \Lie(A) \]
  between the two logarithms
  factors as
  \[ E(\C_p) \xrightarrow\trop N_\Q \overset{L}\To\Lie(A), \]
  where $L$ is $\Q$-linear.
\end{proposition}

We have the following interpretation of Proposition~\ref{prop:compare.integrals}
in terms of paths.  Let $\gamma\colon [0,1]\to A^\an$ be a path with $\gamma(0) = 0$
and $\gamma(1) = x\in A(\C_p)$.  Let $\td\gamma\colon [0,1]\to E^\an$ be the unique
lift starting at $0$, and let $\td x = \td\gamma(1)\in E(\C_p)$.  Then for
$\omega\in\Omega_{A/\C_p}^1(A)$, we have
\begin{equation}\label{eq:integral.difference}
  \presuper\BC\int_\gamma\omega - \presuper\Ab\int_\gamma\omega =
  \angles{L\circ\trop(\td x),\,\omega}, 
\end{equation}
where $\angles{\scdot, \scdot}$ denotes the pairing between $\Lie(A)$ and
$\Omega_{A/\C_p}^1(A)$.
Because $\log_{\Ab}$ vanishes on $\trop(M')\subset E(\C_p)$, the homomorphism
$L$ is uniquely determined by
\[ \angles{L\circ\trop(\td x),\, \omega} = \presuper{\BC}{\int}_\gamma \omega
\sptxt{for} \gamma\in\pi_1(A^\an,0). \]

\begin{proof}[Proof of Proposition~\ref{prop:compare.integrals}]
  Since $A_0(\C_p) = \ker(\trop\colon E(\C_p)\surjects N_\Q)$, the proof amounts to
  showing that $\log_\BC = \log_\Ab$ on $A_0(\C_p)$.  According
  to~\cite{Zarhin:abelian-integrals}, the abelian integral on an abelian
  $\C_p$-Lie group $G$ exists and is characterized by~\eqref{eq:tangent.cond}
  whenever $G$ has the property that $G/U$ is a torsion group for all open
  subgroups $U$ of $G$.  This property is satisfied by $A(\C_p)$.  Since
  $A_0(\C_p)$ is an analytic domain in $A(\C_p)$, it is an open subgroup of
  $A(\C_p)$ in the na\"ive analytic topology, so the property is also satisfied
  by $A_0(\C_p)$.  Hence $\log_\Ab|_{A_0(\C_p)}$ is characterized by the fact
  that it induces the identity map on tangent spaces.

  On the other hand, $A_0$ is simply connected---the deformation retraction of
  $A^\an$ onto $\Sigma$ takes $A_0$ onto $\{0\}$---so the Berkovich--Coleman
  integral on $A_0(\C_p)$ is path-independent.  Hence it suffices to show that
  $\log_\BC$ induces the identity map on $\Lie(A_0) = \Lie(A)$. But $0\in A$ has
  a neighborhood $U$ isomorphic to an open unit polydisc, so $\log_\BC$ can be
  calculated on $U$ by formal antidifferentiation as in Section~\ref{s:twointegrals}.
\end{proof}

Because $N=0$ for abelian varieties of good reduction, we have the following.

\begin{corollary}\label{cor:same.integrals}
  The Berkovich--Coleman and abelian integrals coincide on abelian varieties of
  good reduction.
\end{corollary}

\begin{remark}
  Given a branch of the logarithm, Zarhin~\cite{Zarhin:abelian-integrals}
  defines an abelian integration theory for any commutative $\C_p$-algebraic
  group $G$.  The proof of Proposition~\ref{prop:compare.integrals} shows that
  the Berkovich--Coleman integral coincides with Zarhin's integral on any $G$
  such that $G^\an$ is simply connected and admits a neighborhood of $1$ which
  is isomorphic to a unit polydisc.
\end{remark}

\begin{proof}[Proof of Proposition~\ref{prop:abelian.int}]
  The only part of Definition~\ref{d:it} that does not follow immediately from
  the definitions is condition~(\ref{i:ftc}), the fundamental theorem of
  calculus on open polydiscs.  Let $U\subset A^\an$ be an open subdomain
  isomorphic to an open polydisc.  As $U$ is simply connected, it lifts to an
  open subdomain $\td U\subset E^\an$ which maps isomorphically onto $U$.  By
  Proposition~\ref{prop:compare.integrals}, it suffices to show that
  $\trop(\td U)$ is a single point.

  Choosing a basis for $N$, we can think of $\trop$ as a map $E^\an\to\R^n$.  As
  explained in the paragraph after the statement in~\cite[Theorem~1.2]{bosch_lutkeboh91:degenerat_abelian_varieties}, the extensions
  $0\to T_0\to A_0\to B^\an\to 0$ and $0\to T\to E^\an\to B^\an\to 0$ split
  locally on $B^\an$ in a compatible way.  It follows that the coordinates of
  $\trop$ locally (on $B^\an$) have the form $-\log|f|$ for $f$ an invertible
  function.  Therefore the claim is a consequence of
  Lemma~\ref{lem:local.val.const} below.
\end{proof}

\begin{lemma}\label{lem:local.val.const}
  Let $F\colon \bB^n(1)_+\to\R$ be a continuous function which locally has the form
  $F(x) = -\log|f(x)|$ for an invertible function $f$.  Then $F$ is constant.
\end{lemma}

\begin{proof}
  As $\bB^n(1)_+$ is covered by closed polydiscs of smaller radius, it suffices
  to prove the lemma for the closed polydisc $\bB^n(1)$ instead.  First we
  prove the lemma when $n=1$.  Since $F$ is locally of the form $-\log|f|$, it
  is \defi{harmonic} in the sense
  of~\cite[Definition~5.14]{baker_payne_rabinoff13:analytic_curves}:
  this follows from the slope formula \cite[Theorem~5.15]{baker_payne_rabinoff13:analytic_curves} and
  the fact that harmonicity is a local condition.  Therefore the mean value
  theorem applies, so $F$ attains its maximum on the Shilov boundary point
  $\zeta$ of $\bB(1)$.  By the same reasoning as applied to $-F$, we see that
  $F$ also attains its minimum on $\zeta$.  Thus $F$ is constant.

  The general case follows from the above and these observations: (a) any two
  $\C_p$-points of $\bB^n(1)$ are in the image of a morphism
  $\bB(1)\to\bB^n(1)$, and (b) the $\C_p$-points of $\bB^n(1)$ are dense in
  $\bB^n(1)$ by~\cite[Proposition~2.1.15]{berkovic90:analytic_geometry}.
\end{proof}

We extract the following statement from the proof of
Proposition~\ref{prop:abelian.int}.

\begin{proposition}\label{prop:ball.retracts.to.point}
  Let $\phi\colon \bB^n(1)_+\to E^\an$ be a morphism.  Then
  $\trop\circ\,\phi\colon \bB^n(1)_+\to N_\R$ is constant.
\end{proposition}

\subsection{Integration on curves of any reduction type}
Fix a smooth, proper, connected $\C_p$-curve $X$ of genus at least $1$, let $J$
be its Jacobian, and let $\iota\colon X\injects J$ be the Abel--Jacobi map
defined with respect to a choice of basepoint $x_0\in X(\C_p)$.  Note that
$\iota^*\colon \Omega^1_{J/\C_p}(J)\to\Omega^1_{X/\C_p}(X)$ is an isomorphism
which does not depend on the choice of $x_0$.

As $X^\an$ is a smooth analytic space, it has a Berkovich--Coleman integration
theory $\presuper\BC\int$ as explained in Section~\ref{sec:bc.integral}.

\begin{definition}
  The \defi{abelian integral} on $X^\an$ is the map
  $\presuper\Ab\int\colon \calP\times\Omega_{X/\C_p}^1(X)\to\C_p$ defined by
  \[ \presuper\Ab\int_\gamma\iota^*\omega =
  \presuper\Ab\int_{\iota\circ\gamma}\omega. \]
\end{definition}

See~\cite{Zarhin:abelian-integrals} for a much more general construction along
these lines.

\begin{lemma}\label{lem:abelian.int.curves}
  The abelian integral is an integration theory on $X^\an$ in the sense
  of Definition~\ref{d:it}, which is independent of the choice of basepoint $x_0$.
\end{lemma}

\begin{proof}
  The only statement that does not follow immediately from the definitions is
  condition~(\ref{i:ftc}), which is a consequence of
  Proposition~\ref{prop:ball.retracts.to.point}. 
\end{proof}

By Corollary~\ref{cor:same.integrals}, the Berkovich--Coleman and abelian
integrals coincide on $X^\an$ when $J$ has good reduction and, in particular,
when $X$ has good reduction.  In this rest of this
section, we will make explicit the difference between the
Berkovich--Coleman and abelian integrals on a basic wide open subdomain in terms
of the tropical Abel--Jacobi map.

\subsubsection{The tropical and algebraic Abel--Jacobi maps}
Let $\Gamma\subset X^\an$ be a skeleton of $X$ with retraction map
$\tau\colon  X^\an\to\Gamma$, as in Section~\ref{sec:skeletons}.  The \defi{Jacobian} of the
metric graph $\Gamma$ is the quotient
$J(\Gamma) = \Div^0(\Gamma)/\Prin(\Gamma)$, where $\Div^0(\Gamma)$ is the group
of degree-zero divisors in $\Gamma$ and $\Prin(\Gamma)$ is the subgroup of
divisors of meromorphic functions on $\Gamma$ (see Section~\ref{sec:slope.formula}).
The Jacobian of $\Gamma$ is a real torus and is moreover a principally
polarized tropical abelian variety in the sense of~\cite[Section~3.7]{baker_rabinoff}.
Fixing a basepoint $P_0\in\Gamma$, we define the
\defi{tropical Abel--Jacobi map} $\beta\colon \Gamma\to J(\Gamma)$ by the usual formula
\[ \beta(P) = [(P)-(P_0)]. \]

Let $\Sigma = N_\R/\trop(M')$ be the skeleton of $J^\an$, and let
$\bar\tau\colon J^\an\to\Sigma$ be the retraction map, as
in Section~\ref{sec:uniformization}.

\begin{theorem}[Baker and Rabinoff~{\cite[Theorem~2.9, Proposition~5.3]{baker_rabinoff}}]\label{thm:skeleton.jacobian}
  There is a canonical isomorphism $\Sigma\isomto J(\Gamma)$ making the
  following square commute:
  \[\xymatrix @=.2in{
    {\Div^0(X)} \ar[d]_{\tau_*} \ar[rr] & & {J^\an} \ar[d]^{\bar\tau} \\
    {\Div^0(\Gamma)} \ar[r] & {J(\Gamma)} & {\Sigma} \ar[l]_(.4)\sim
  }\]
  In other words, for $D\in\Div^0(X)$, the point $\bar\tau([D])\in\Sigma$ is
  identified with $[\tau_*D]\in J(\Gamma)$.
\end{theorem}

From now on we will implicitly identify $J(\Gamma)$ with $\Sigma$.
In~\cite[Proposition~6.2]{baker_rabinoff} it is shown that
Theorem~\ref{thm:skeleton.jacobian} is compatible (under retraction) with the algebraic and
tropical Abel--Jacobi maps.

\begin{proposition}[Baker and Rabinoff~{\cite[Proposition~6.2]{baker_rabinoff}}]
\label{prop:compat.aj}
  Fix $x_0\in X(\C_p)$ and $P_0 = \tau(x_0)\in\Gamma$, and let $\iota\colon X\to J$
  and $\beta\colon \Gamma\to\Sigma$ be the corresponding Abel--Jacobi maps.  Then the
  following square commutes:
  \[\xymatrix @=.2in{
  {X^\an} \ar[r]^\iota \ar[d]_\tau & {J^\an} \ar[d]^{\bar\tau} \\
  {\Gamma} \ar[r]_\beta & {\Sigma}
  }\]
\end{proposition}

From now on we assume that the algebraic and tropical Abel--Jacobi maps are
taken with respect to compatible basepoints as in
Proposition~\ref{prop:compat.aj}.  Let $V\subset\Gamma$ be a simply connected
open subgraph with edge lengths in $\Q$, and let $U = \tau^{-1}(V)$, an open
analytic domain in $X^\an$.  For example, $U$ could be a basic wide open
subdomain.  As $U$ is simply connected as well, the restriction of the
Abel--Jacobi map $\iota\colon  X^\an\to J^\an$ to $U$ lifts uniquely through the universal cover
$\pi\colon E^\an\to J^\an$ to a morphism $\td\iota\colon  U\to E^\an$ taking the
basepoint to the origin.  Since
$U\setminus V$ is a disjoint union of open discs, each retracting to a unique
point of $V$, by Proposition~\ref{prop:ball.retracts.to.point} the composition
$\trop\circ\,\td\iota\colon U\to N_\R$ factors through the retraction to the skeleton
$\tau\colon U\to V$.  Moreover, by Proposition~\ref{prop:compat.aj} the restriction
$\td\beta$ of $\trop\circ\,\td\iota$ to $V$ is a lift of the restriction of the
tropical Abel--Jacobi map $\beta\colon \Gamma\to\Sigma$ to $V$ through the universal
covering map $N_\R\to\Sigma$.  In summary, the following diagram is commutative:
\begin{equation}\label{eq:lift.aj.diagram}
\xymatrix{
  {U} \ar[r]_{\td\iota} \ar@/^1.5pc/[rr]^\iota \ar[d]_\tau &
  {E^\an} \ar[r]_\pi \ar[d]^{\trop} & {J^\an} \ar[d]^{\bar\tau} \\
  {V} \ar[r]^{\td\beta} \ar@/_1.5pc/[rr]_\beta &
  {N_\R} \ar[r] & {\Sigma}
}\end{equation}

The following result of
Mikhalkin and Zharkov, which is a consequence of the discussion in Section~6 of \cite{mikhalkin_zharkov08:tropical_curves}
(see also~\cite[Theorem~4.1]{baker_faber11:metric_properties_tropical_abel_jacobi}), 
says that the map $\td\beta\colon V\to N_\R$ is very well behaved.

\begin{theorem}[Mikhalkin and Zharkov~{\cite[Section~6]{mikhalkin_zharkov08:tropical_curves}}]\label{thm:tropical.aj.is.nice}
  The partial lift $\td\beta\colon V\to N_\R$ of $\beta\colon \Gamma\to\Sigma$ satisfies the
  following properties.
  \begin{enumerate}
  \item If $e\subset V$ is an edge such that $\Gamma\setminus e$ is
    disconnected, then $\td\beta$ is constant on $e$.
  \item If $e\subset V$ is an edge such that $\Gamma\setminus e$ is
    connected, then $\td\beta$ is affine-linear on $e$ with rational slopes.
  \item Vertices of $V$ map into $N_\Q$.
  \item $\td\beta$ satisfies the tropical balancing condition.
  \end{enumerate}
\end{theorem}

The balancing condition in the last part of
Theorem~\ref{thm:tropical.aj.is.nice} roughly says that at any vertex $v\in V$,
a weighted sum of the images of the tangent vectors at $v$ under $\td\beta$ is
equal to zero.  This implies, for instance, that if $v$ has three adjacent edges
$e_1,e_2,e_3$, then their images under $\td\beta$ are coplanar (see the end of
~\cite[Section~3]{baker_faber11:metric_properties_tropical_abel_jacobi} for details).

\subsubsection{Comparison of the integrals, bis}
Now we are able to draw some consequences for integration on basic wide open
subdomains.
Suppose that $V$ is an open star neighborhood of a type-$2$ point
$\zeta\in\Gamma$ as in Section~\ref{sec:basic.wide.open}, so that $U = \tau^{-1}(V)$
is a basic wide open subdomain.  Recall that $\deg(\zeta)$ denotes the valency
of $\zeta$ as a vertex in $\Gamma$, which is at least $1$ since a basic wide
open by definition has at least one end.

\begin{proposition}\label{prop:correction.wide.open}
  Let $H\subset\Omega_{X/\C_p}^1(X)$ be the subspace of those $1$-forms $\omega$ such
  that $\presuper\BC\int_\gamma\omega = \presuper\Ab\int_\gamma\omega$ for all
  paths $\gamma\colon [0,1]\to U$ with endpoints in $U(\C_p)$.  Then the codimension
  of $H$ in $\Omega_{X/\C_p}^1(X)$ is strictly less than $\deg(\zeta)$.
\end{proposition}

\begin{proof}
  We are free to choose the basepoint $x_0$ in $U(\C_p)$ such that
  $P_0 = \tau(x_0) = \zeta$.  We choose the lift $\td\iota\colon U\to E^\an$ of
  $\iota$ such that $\td\iota(x_0) = 0$.  Since we can compose paths, we have
  $\omega\in H$ if and only if
  $\presuper\BC\int_\gamma\omega = \presuper\Ab\int_\gamma\omega$ for all paths
  $\gamma$ such that $\gamma(0) = x_0$.  As $U$ is simply connected, the
  Berkovich--Coleman integral is path-independent, so we write
  $\presuper\BC\int_{x_0}^{x}\omega = \presuper\BC\int_{\gamma}\omega$ for any
  path $\gamma$ from $x_0$ to $x\in U(\C_p)$.

  By Proposition~\ref{prop:compare.integrals}
  and~\eqref{eq:integral.difference}, there is a linear map
  $$L\colon N_\Q\to\Lie(J) = \Hom(\Omega_{X/\C_p}^1(X),\C_p)$$ such that for all
  $\omega\in\Omega_{X/\C_p}^1(X)$ and all $x\in U(\C_p)$, we have
  \[ \presuper\BC\int_{x_0}^x\omega - \presuper\Ab\int_{x_0}^x\omega =
  \angles{L\circ\trop(\td\iota(x)),\,\omega}. \]
  By the balancing condition in Theorem~\ref{thm:tropical.aj.is.nice},
  $\trop(\td\iota(U(\C_p))) = \td\beta(V)\cap N_\Q$ spans a $\Q$-vector space
  of dimension at most $\deg(\zeta)-1$ (note that $\td\beta(\zeta) = 0$ since
  $\zeta$ is the basepoint of the tropical Abel--Jacobi map $\beta$).  Therefore
  the annihilator $H$ of $L(\trop(\td\iota(U(\C_p))))$ has dimension strictly
  less than $\deg(\zeta)$.
\end{proof}

If $V\subset\Gamma$ is an open edge, then the open annulus $U = \tau^{-1}(V)$ is
a basic wide open subdomain with respect to any type-$2$ point $\zeta\in V$.
In this case one has the following slightly more precise variant of
Proposition~\ref{prop:correction.wide.open}, recovering a result of
Stoll~\cite[Proposition~7.3]{Stoll:uniformity}.

\begin{proposition}\label{prop:correction.annulus}
  Let $e\subset\Gamma$ be an open edge, and let $A = \tau^{-1}(e)\subset X^\an$,
  an open annulus.  Choose an identification $A\cong\bS(\varpi)_+$ with the standard
  open annulus of inner radius $|\varpi|$ and outer radius $1$.  Then for all
  $\omega\in\Omega_{X/\C_p}^1(X)$ there exists $a(\omega)\in\C_p$ such that
  \[ \presuper\BC\int_{x}^y\omega - \presuper\Ab\int_{x}^y\omega =
  a(\omega)\big(\val(y) - \val(x)\big) \]
  for all $x,y\in\bS(\varpi)_+(\C_p)$.  Moreover, $\omega\mapsto a(\omega)$ is
  $\C_p$-linear.
\end{proposition}

The proof is almost identical to that of Proposition~\ref{prop:correction.wide.open} and is left to the reader.

\begin{remark}
We expect that the above results should make it possible to compute abelian integrals on hyperelliptic curves of bad reduction in residue characteristic greater than $2$.  Such curves have an explicit cover by hyperelliptic wide opens that can be obtained from their defining equations (see ~\cite{Stoll:uniformity}).  The Balakrishnan--Bradshaw--Kedlaya algorithm \cite{BBK} can be applied to such wide opens to compute Berkovich--Coleman integrals.  After determining the tropical Abel--Jacobi map through the use of tropical $1$-forms (see~\cite{mikhalkin_zharkov08:tropical_curves}), one can then obtain the abelian integrals.
\end{remark}

\section{Bounding zeros of integrals on wide opens}
\label{sec:bounds-zeroes}

In this section, we leverage Proposition~\ref{prop:zeros.on.wide.open} to bound
the number of zeros of the Berkovich--Coleman integral of an \emph{exact}
$1$-form $\omega = df$ on a basic wide open curve.  This amounts to relating the
slopes of $-\log\|\omega\|$ to those of $-\log|f|$ on an annulus, which we do
in Proposition~\ref{p:annularbounds}.  To eventually obtain bounds depending
essentially only on the genus, we will also need a combinatorial argument about
stable metric graphs, which we make in Lemmas~\ref{lem:graph.combinatorics} and~\ref{lem:unstable.combinatorics}.  The
main result of the section is Theorem~\ref{thm:wideopen.bounds}.

In this section we work over $K = \C_p$.

\subsection{Slopes on annuli}\label{sec:slopes.on.annuli}
First, we recall the relationship between Newton polygons and slopes on the
skeleton of an annulus.  Let $\varpi\in\C_p^\times$ with $|\varpi| < 1$, and
recall that $\bS(\varpi)_+$ denotes the open annulus of outer radius $1$ and
inner radius $|\varpi|$.  Let $a=\val(\varpi)$, the logarithmic modulus of
$\bS(\varpi)_+$.  An analytic function on $\bS(\varpi)_+$ can be expressed as an
infinite-tailed Laurent series $\sum_{n\in\Z} a_n T^n$ with the property that
\[ \val(a_n) + nr \to \infty \sptxt{as} n\to\pm\infty \]
for all $r\in(0,a)$.
For $r\in(0,a)$, we set
\[ \bigg\|\sum a_n T^n\bigg\|_r = \max\big\{ |a_n|\exp(-nr) \big\}. \]
This is a multiplicative seminorm which defines a point $\xi_r\in\bS(\varpi)_+$.
The map $\sigma\colon(0,a)\to\bS(\varpi)_+$ given by $\sigma(r) = \xi_r$
is a continuous embedding and its image
$\Sigma(\bS(\varpi)_+) \coloneq \sigma((0,a))$ is by definition the skeleton of
$\bS(\varpi)_+$. 

Note that if $f(T) = \sum a_n T^n$ is an analytic function on $\bS(\varpi)_+$,
$F = -\log|f|$, and $\xi_r = \sigma(r)\in\Sigma(\bS(\varpi)_+)$, then by
definition
\begin{equation}\label{eq:F.of.xir}
  F(\xi_r) = -\log\|f\|_r = \min\{\val(a_n) + nr ~:~ a_n\neq 0\}.
\end{equation}

See Figure~\ref{fig:newton.poly.annulus} for an illustration in terms of Newton polygons.

\begin{figure}[ht]
  \centering
  \begin{tikzpicture}[every circle/.style={fill=black, radius=.7mm}, scale=.8,
    help lines/.style={black!20!white, line width=.2pt},
    every node/.style={font=\tiny, fill=white, inner sep=0pt}]
    \draw [help lines] (-10,-2) grid (2,4);
    \draw [thick, ->] (-10, 0) -- (2,0);
    \draw [thick, ->] (0, -2) -- (0,4);
    \draw [dashed] (-10, 4) -- (2,-2)
          node [pos=.1, below=2mm] {$r = 1/2$};
    \filldraw (-8.3, 4) -- (-6, 2) circle
           -- (-4, 1) circle -- (1, -.5) circle -- (2,-.6);
    \fill (-7, 3.5) circle[] (-5, 3) circle[] (-3, 2) circle[]
          (-2, 1.5) circle[] (-1, 3) circle[] (0, 2) circle;
    \node at (-4, 0) [below=2mm] {$-4$};
    \node at (0,-1) [left=2mm] {$-1$};
  \end{tikzpicture}
  
  \caption{A possible Newton polygon of an analytic function $f = \sum a_nT^n$
    on an open annulus.  The dashed line is $y + \frac 12x = -1$.  If
    $F = -\log|f|$, then $F(\xi_{1/2}) = -1$, and $d_v F(\xi_{1/2}) = 4$ in the
    notation of Lemma~\ref{lem:slope.from.np}.}
  \label{fig:newton.poly.annulus}
\end{figure}
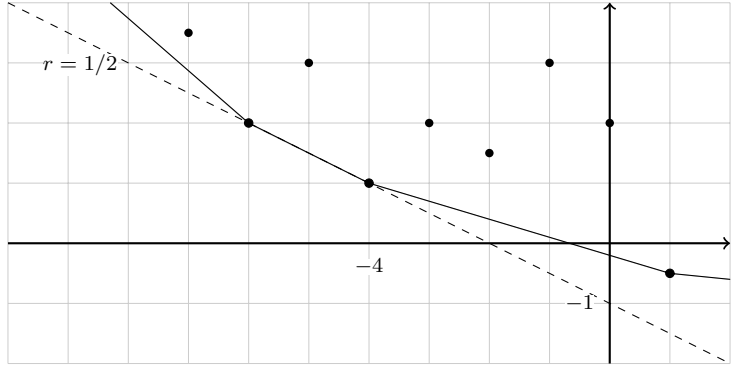

Recall that if $v$ is a tangent direction at $\xi_r$, then
$d_v F(\xi_r)$ denotes the slope of $F$ in the $v$ direction.

\begin{lemma}\label{lem:slope.from.np}
  Let $\xi_r = \sigma(r)\in\Sigma(\bS(\varpi)_+)$, and let $v$ be the tangent
  direction at $\xi_r$ defined by the line segment $\sigma((0,r])$.  Let
  $f(T) = \sum a_n T^n$ be an analytic function on $\bS(\varpi)_+$, and let
  $F = -\log|f|$.  Then
  \[ d_v F(\xi_r) = -\max\big\{ n~:~ \val(a_n) + nr = F(\xi_r) \big\}. \]
\end{lemma}

\begin{proof} Let $N = \max\{n~:~\val(a_n) + nr = F(\xi_r)\}$.  There exists a small
$\epsilon$ such that $\val(a_N) + Ns < \val(a_n) + ns$ for all $n\neq N$ and all
$s\in(r-\epsilon, r)$.  It follows that the restriction of $f$ to the
subannulus
\[ A = \big\{ \eta\in\bS(\varpi)_+~:~-\log|T(\eta)|\in(r-\epsilon,r) \big\} \]
is invertible, with $|f(\eta)| = |a_N\eta^N|$ for all $\eta\in A$.  Therefore
the slope of $-\log|f|$ along $\sigma((r-\epsilon,r))$ (in the \emph{positive}
direction) is equal to $N$
by~\cite[Proposition~2.5(1)]{baker_payne_rabinoff13:analytic_curves}.
\end{proof}

All of our bounds will be stated in terms of the following function $N_p(r,N_0)$.

\begin{definition}\label{def:Np}
  Let $r$ be a positive real number, let $N_0$ be an integer, and let $p$ be a
  prime.  Define $N_p(r,N_0)$ to be the smallest positive integer $N$ such that
  for all $n\geq N$, one has
  \begin{equation}\label{eq:one.annulus.bound}
    r(n - N_0) > \lfloor\log_p(n)\rfloor.
  \end{equation}
\end{definition}

\begin{remark}
\label{R:bounds-on-NpAB}
  The integer $N_p(r,N_0)$ gets larger as $N_0$ increases and as $r$
  decreases, and it gets smaller as $p$ increases.  Clearly
  \[ \lim_{s\nearrow r} N_p(s,\,N_0) = N_p(r,\,N_0). \]
  If $N_0\geq 0$ and $p \geq N_0 + 2$, then $N_p(r,N_0) = N_0 + 1$ because
  $\lfloor\log_p(N_0+1)\rfloor = 0 < r$.  One should think of $N_p(r,N_0)-N_0$
  as the correction to the $p$-adic Rolle theorem coming from the fact that
  $1/p$ has negative valuation (see Corollary~\ref{cor:disc.bound} and see also~\cite[Section~6]{stoll06:rational_points} for a more sophisticated approach to
  the same problem). (Stoll's correction factor $\delta(\scdot,\scdot)$ is
  slightly better, but ours is easier to define.)
 
  We give an explicit upper bound on $N_p(r,N_0)$ as follows.  If $N_0\leq 0$, we
  take $N = 1$.  Otherwise, write $N=N_0\exp(u)$ for $u>0$.  We want
  $N-N_0>\frac{1}{r \ln(p)}\ln(N)$; that is,
\[\exp(u)-1>\frac{\ln(N)}{N_0r\ln(p)}=\frac{\ln(N_0)}{N_0r\ln(p)}+\frac{u}{N_0r\ln(p)}.\]
  Writing $\exp(u)-1>u+\frac{u^2}{2}$ and using $u$ and $u^2$ to bound each term on the right, it suffices to pick
  \[u \geq \max\left(\frac{\ln(N_0)}{N_0r\ln(p)},\,\frac{2}{N_0r\ln(p)}\right).\]
  For $N_0\leq 7$, this gives 
  \[N_p(r,N_0)\leq \lceil N_0\exp(u) \rceil=\left\lceil N_0\exp\left(\frac{2}{N_0r\ln(p)}\right)\right\rceil,\]
  while for $N_0\geq 8$, we have
  \[N_p(r,N_0)\leq \lceil N_0\exp(u) \rceil= \left\lceil
    N_0^{1+1/(N_0r\ln(p))}\right\rceil.\]
  If we suppose that $r\ln(p)\geq 1$, then one checks case by case that
  \begin{equation}\label{eq:nice.NpAB.bound}
    N_p(r,N_0) \leq 2N_0
  \end{equation}
  for all $N_0\geq 1$.
\end{remark}

In the statement of the next proposition we will use the following notation (see Figure~\ref{fig:annulus.setup}):

\begin{longtable}[l]{rl}
  \hspace{1em} $X$ & A smooth, proper, connected $\C_p$-curve. \\
  $\fX$ & A semistable $\sO_{\C_p}$-model of $X$. \\
  $\Gamma$ & $= \Gamma_\fX\subset X^\an$, a skeleton of $X$ in the sense of
  Section~\ref{sec:skeletons}. \\
  $\bar e$ & $\subset\Gamma$, a closed interval with type-$2$ endpoints. \\
  $\zeta_\pm$ & The endpoints of $\bar e$. \\
  $v_\pm$ & The tangent direction at $\zeta_\pm$ in the direction of $e$. \\
  $e$ & $= \bar e\setminus\{\zeta_\pm\}$, the open interval inside $e$. \\
  $A$ & $= \tau^{-1}(e) \cong \bS(\varpi)_+$, an open annulus. \\
  $a$ & $= \val(\varpi)$, the logarithmic modulus of $A$.
\end{longtable}

\vskip -.13in
\noindent We choose an identification $A\cong\bS(\varpi)_+$ such that
$\xi_r\to\zeta_-$ as $r\to 0$, so $\xi_r\to\zeta_+$ as $r\to a$.
If $e$ is a loop edge, then $\zeta_+ = \zeta_-$, and we define $v_\pm$ to be the
two tangent directions at $\zeta_\pm$ in the direction of $e$.
In what
follows we use the formal metric $\|\scdot\|$ on $\Omega^1_{X/\C_p}$ induced by
the sheaf of integral Rosenlicht differentials on $\fX$, as in
Section~\ref{sec:rosenlicht}.

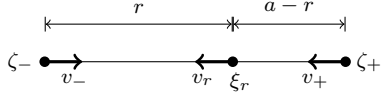
\begin{figure}[ht]
  \centering
  \begin{tikzpicture}[every circle/.style={fill=black, radius=.7mm},
    every node/.style={font=\tiny}]
    \coordinate (ZM) at (0,0);
    \coordinate (ZP) at (4,0);
    \coordinate (R) at (2.5, 0);
    \draw (ZM) node[left] {$\zeta_-$} -- (ZP) node[right] {$\zeta_+$};
    \draw[->, very thick] (ZM) -- ++(.5cm,0)
          node [pos=0.8, below] {$v_-$};
    \draw[->, very thick] (ZP) -- ++(-.5cm,0)
          node [pos=0.8, below] {$v_+$};
    \fill (ZM) circle[] (ZP) circle[] (R) circle;
    \path (R) ++(1mm,0) node [below] {$\xi_r$};
    \draw[->, very thick] (R) -- ++(-.5cm,0) node[pos=.8, below] {$v_r$};
    \draw[|<->|, thin] ($(ZM)+(0,5mm)$) -- ($(R)+(0,5mm)$) node[pos=.5, above] {$r$};
    \draw[|<->|, thin] ($(R)+(0,5mm)$) -- ($(ZP)+(0,5mm)$) node[pos=.5, above] {$a-r$};
  \end{tikzpicture}
  
  \caption{Illustration of the notation used in
    Proposition~\ref{p:annularbounds}.  The interval represents the edge $e$,
    which has length $a$.}
  \label{fig:annulus.setup}
\end{figure}

\begin{proposition}\label{p:annularbounds} 
  With the above notation, let $\omega\in H^0(X,\Omega^1_{X/\C_p})$ be a nonzero
  global differential, and suppose that $\omega$ is~\emph{exact} on $A$, so
  $\omega = df$ for an analytic function $f$ on $A$.  Let $F = -\log|f|$ and
  $F_0 = -\log\|\omega\|$, and let $N_0 = d_{v_+}F_0(\zeta_+)$.  Choose
  $r\in(0,a)$, and let $v_r$ be the tangent direction at $\xi_r$ in the direction
  of $\zeta_-$.
  Then $d_{v_r} F(\xi_{r})\leq N_p(a-r,N_0)$.
\end{proposition}

\begin{proof}
  Let $T\colon A\isomto\bS(\varpi)_+$ be the identification we chose above, so
  $\xi\to\zeta_-$ as $-\log|T(\xi)|\to 0$.  The
  restriction of $\omega$ to $A$ has an infinite-tailed Laurent
  series expansion of the form
  \[ \omega = \sum_{n\in\Z} a_n T^n\,\frac{dT}T. \]
  By Lemmas~\ref{lem:nonv.sec.annulus} and~\ref{lem:indep.of.skel}, 
  for $\xi\in A$ one has
  $F_0(\xi) = -\log\|\omega(\xi)\| = -\log|\sum a_n\xi^n|$.  
  Using~\eqref{eq:F.of.xir} and taking the limit as $r\to a$, we obtain
\begin{equation}\label{eq:log.differential.norm}
\begin{split}
  F_0(\zeta_+) = -\log\|\omega(\zeta_+)\| &= \inf\{\val(a_n) + na~:~a_n\neq 0\}.
\end{split}
\end{equation}
In particular, the right-hand side of this equation is finite.  Since $\omega$
has finitely many zeros on $A$, the Newton polygon $\mathcal{N}$ of
$\sum a_n T^n$ has finitely many segments with slope in $(-a, 0)$.  Therefore
the infimum in~\eqref{eq:log.differential.norm} is achieved, and
$d_{v_+} F_0(\zeta_+) = d_{v_{r'}} F_0(\xi_{r'})$ for $r' < a$ very close to
$a$, where $\xi_{r'}$ and $v_{r'}$ are defined as in the statement of the
proposition.  From this and Lemma~\ref{lem:slope.from.np}, as applied to
$\sum a_n T^n$ and $\xi_{r'}$ with $r'\to a$, one sees that
\begin{equation}\label{eq:calc.N0}
  N_0 = -\max\{n~:~ \val(a_n) + na = F_0(\zeta_+) \}.
\end{equation}

Since $df = \omega$, we have
\[ f = \sum_{n\in\Z} b_n T^n = b_0 + \sum_{n\neq 0} \frac{a_n}n T^n \]
on $A$, where $b_n = a_n/n$ for $n\neq 0$ and $b_0\in\C_p$ is some constant.
According to Lemma~\ref{lem:slope.from.np}, 
\begin{equation}\label{eq:calc.dvF}
  d_{v_r} F(\xi_{r}) = -\max\big\{ n~:~ \val(b_n) + nr = F(\xi_{r}) \big\},
\end{equation}
where
\begin{equation}\label{eq:calc.Fxir}
  F(\xi_r) = \min\{\val(b_n) + nr ~:~ b_n\neq 0\}.
\end{equation}

The number $N \coloneq N_p(a-r, N_0)$ is positive, so if $d_{v_r} F(\xi_r)\leq 0$, then we
are done.  Hence, we may assume $d_{v_r} F(\xi_r) > 0$, so that
$\val(b_n) + nr = F(\xi_r)$ implies $n < 0$.  Note that we are in a situation
where the constant $b_0$ plays no role.  For $n < 0$ such that $a_n\neq 0$, we
have
\[\begin{split}
  \val(b_n) + nr &= \val(a_n) + na - \val(n) - n(a-r) \\
  &\geq \val(a_{-N_0}) - N_0a -\val(n) - n(a-r) \\
  &= \val(a_{-N_0}) - N_0r - \val(n) - (n+N_0)(a-r) \\
  &\geq \val(a_{-N_0}) - \val(-N_0) - N_0r - \val(n) - (n+N_0)(a-r) \\
  &= \val(b_{-N_0}) - N_0r - \val(n) + (-n-N_0)(a-r) \\
  &\geq F(\xi_r) - \lfloor\log_p(-n)\rfloor + (-n-N_0)(a-r).
\end{split}\]
Here we have used~\eqref{eq:log.differential.norm} and~\eqref{eq:calc.N0} in the
first inequality, and~\eqref{eq:calc.dvF} and~\eqref{eq:calc.Fxir} in the last
(along with $\val(n)\leq\lfloor\log_p(-n)\rfloor$).  It follows that
when~\eqref{eq:one.annulus.bound} is satisfied, then
$\val(b_n) + nr > F(\xi_r)$ for $n\leq -N$, so that $N \geq d_{v_r} F(\xi_r)$.
\end{proof}

We would like to apply Proposition~\ref{p:annularbounds} to arbitrary open
annuli embedded in $X^\an$.  For this we need the following lemma.

\begin{lemma}\label{lem:skel.from.annulus}
  Let $U\subset X^\an$ be an open subdomain isomorphic to an open annulus
  $\bS(\varpi)_+$.  Then there exists a skeleton $\Gamma$ of $X$ and an open
  edge $e$ of $\Gamma$ such that $\tau^{-1}(e) = U$.
\end{lemma}

\begin{proof}
  First we recall that if $U'\subset X^\an$ is an open subdomain isomorphic to
  the open disc $\bB(1)_+$, then the closure of $U'$ is $U'\djunion\{x\}$ for a
  type-$2$ point $x\in X^\an$
  by~\cite[Lemma~3.3]{abbr14:lifting_harmonic_morphism_I}.

  By~\cite[Lemma~3.6]{abbr14:lifting_harmonic_morphism_I}, the closure of $U$ in
  $X^\an$ is $U\djunion\{x,y\}$, where $x,y\in X^\an$ are points which are not
  necessarily distinct.  We claim that $x,y$ have type $2$.  This claim reduces
  to the case of a disc by doing surgery on $X^\an$, as in the proof
  of~\cite[Lemma~4.12(2)]{baker_payne_rabinoff13:analytic_curves}.  Briefly, one
  excises a closed subannulus from $U$, then caps the ends of the remaining two
  open annuli by open discs.  One obtains a new open set
  $U'\cong\bB(1)_+\djunion\bB(1)_+$ in a new curve $X'^\an$, with $x,y$ identified
  with the points in the closures of the two open discs.

  Let $V$ be a semistable vertex set of $X$ containing $x$ and $y$.  Such exists
  by~\cite[Proposition~3.13(3)]{baker_payne_rabinoff13:analytic_curves}.  Let
  $V' = V\setminus U$.  Then $X^\an\setminus V'$ is again a disjoint union of
  open discs and open annuli, one of which is $U$, so $V'$ is a semistable
  vertex set.  The corresponding skeleton $\Gamma$ has an open edge
  $e\coloneq\Gamma\cap U$ satisfying the conditions of the lemma.
\end{proof}

As an immediate consequence we recover a general version of the standard
Chabauty--Coleman bound for zeros of an antiderivative on an open disc, as found
(in a slightly stronger version) in~\cite[Proposition~6.3]{stoll06:rational_points}.

\begin{corollary}\label{cor:disc.bound}
  Let $B\subset X^\an$ be an open subset isomorphic to the open unit disc
  $\bB(1)_+$, and choose an isomorphism $T\colon B\to\bB(1)_+$.  Let
  $\omega\in H^0(X,\Omega^1_{X/\C_p})$ be a nonzero global differential, and let
  $N_0$ be the number of zeros of $\omega$ on $B$.  Then
  $\omega = df$ for an analytic function $f$ on $B$, and for any $r > 0$, $f$
  has at most $N_p(r,N_0+1)$ zeros on the subdisc
  $B_r \coloneq \{\eta\in B~:~-\log|T(\eta)| > r\}$.
\end{corollary}

\begin{proof}
  That $\omega$ is exact follows from the Poincar\'e lemma.
  Let $g$ be an analytic function on a disc $\bB(1)_+$ with finitely many zeros.
  By a classical Newton polygon argument, the number of zeros of $g$ on
  $\bB(1)_+$ is equal to the slope of $-\log|g|$ at the Gauss point $\zeta_r$ of
  the closed disc of radius $\exp(-r)$ for $r>0$ close to zero.  Hence, the
  corollary follows from Proposition~\ref{p:annularbounds} as applied to an
  annulus of logarithmic modulus $a > r$ contained in $B$, recalling that the
  slope of $\omega$ on an annulus is calculated with respect to $dT/T$.
\end{proof}

\subsection{Combinatorics of stable graphs}\label{sec:stable.graphs}
The minimal skeleton $\Gamma = \Gamma_{\min}$ (in the sense
of Section~\ref{sec:skeletons}) of a curve of genus $g \geq 2$ is the skeleton
associated to a stable model.  This implies that $\Gamma$ is a connected metric
graph, with vertices $x$ weighted by the genus $g(x)$, such that all vertices of
valency $\leq 2$ have positive weight.  Such a metric graph is called
\defi{stable}.

In this subsection we make some (undoubtedly well known) observations about the
combinatorics of stable vertex-weighted metric graphs $(\Gamma, g)$.  We extend
the weight $g$ to all points of $\Gamma$ by setting $g(x) = 0$ if $x$ is not a
vertex.  Likewise, we declare that the valency of a nonvertex $x\in\Gamma$ is
$\deg(x)\coloneq 2$.  The \defi{genus} of $\Gamma$ is defined via the genus
formula~\eqref{eq:genus.formula}: that is,
\[ g(\Gamma) \coloneq h_1(\Gamma) + \sum_{x\in\Gamma} g(x). \]
Recall~\eqref{eq:canon.div.graph} that the \defi{canonical divisor} on $\Gamma$ is 
\[ K_\Gamma \coloneq \sum_{x\in\Gamma} \big(2g(x) - 2 + \deg(x)\big)\,(x). \]
The degree of $K_\Gamma$ is $2g(\Gamma)-2$, and since $\Gamma$ is stable,
$K_\Gamma$ is \emph{effective} and has positive multiplicity on every vertex.

\begin{lemma}\label{lem:graph.combinatorics}
  Let $(\Gamma, g)$ be a stable vertex-weighted metric graph of genus
  $g(\Gamma)\geq 2$.
  \begin{enumerate}
  \item $\Gamma$ has at most $2g-2$ vertices.
  \item $\Gamma$ has at most $3g-3$ edges and at most $g$ loop edges.
  \item Every vertex of $\Gamma$ has valency at most $2g(\Gamma)$.
  \end{enumerate}
\end{lemma}

\begin{proof}
  As mentioned above, the canonical divisor $K_\Gamma$ has degree $2g(\Gamma)-2$
  and is effective, with positive multiplicity on vertices.  Since
  $2g(x) - 2 + \deg(x) = 0$ for $x$ not a vertex, $K_\Gamma$ is supported
  on the set of vertices.  This proves~(1).  Letting $V$ be the number of
  vertices of $\Gamma$ and $E$ be the number of edges, we have
  $h_1(\Gamma) = E - V + 1$, so
  \[ E = h_1(\Gamma) + V - 1 \leq g(\Gamma) + (2g(\Gamma) - 2) - 1 = 3g(\Gamma)
  - 3. \]
  Clearly a graph with more than $g$ loop edges has genus greater than $g$,
  so this proves~(2).  For~(3), note that
  \[ 2g(\Gamma) - 2 = \deg(K_\Gamma) = \sum\big(2g(x) - 2 + \deg(x)\big), \]
  where the sum is taken over all vertices.  Since each summand is positive, for
  a given vertex $x$, we have $2g(x) - 2 + \deg(x) \leq 2g(\Gamma) - 2$, so
  \[ \deg(x) \leq 2g(\Gamma) - 2g(x) \leq 2g(\Gamma). \]
\end{proof}

The following lemma does not require the weighted metric graph to be stable.
It plays the role of~\cite[Corollary~6.7]{Stoll:uniformity},
which is proved using an explicit calculation on hyperelliptic curves.

\begin{lemma}\label{lem:unstable.combinatorics}
  Let $(\Gamma, g)$ be a vertex-weighted metric graph of genus
  $g(\Gamma)$.  Let $F$ be a tropical meromorphic function on $\Gamma$
  such that $\div(F) + K_\Gamma \geq 0$.  Then for all $x\in\Gamma$ and all
  tangent directions $v$ at $x$, we have $|d_v F(x)| \leq 2g(\Gamma)-1$.
  If $K_\Gamma$ is effective, that is, if $\Gamma$ has no genus-zero leaves, then 
  we may replace $2g(\Gamma)-1$ by $2g(\Gamma)-2$.
\end{lemma}

\begin{proof}
  We may assume that $x$ is not a vertex and that $F$ is differentiable at $x$.
  First we assume that $\Gamma$ has no leaves of genus zero, so that $K_\Gamma$
  is effective.  If $F$ is constant in a neighborhood of $x$, then we are done,
  so assume that this is not the case.  Let $r = F(x)$, let
  $\Gamma_{\leq r} = \{y\in\Gamma~:~F(y)\leq r\}$, and define $\Gamma_{<r}$
  similarly.  Then $\Gamma_{\leq r}$ is a subgraph of $\Gamma$, $x$ is a leaf of
  $\Gamma_{\leq r}$, and the tangent direction $v$ at $x$ in which $F$ is
  increasing points away from $\Gamma_{\leq r}$.  Let $x_1,\ldots,x_n$ be the
  points on the boundary of $\Gamma_{\leq r}$ in $\Gamma$, and let $\{v_{ij}\}$
  be the tangent directions at $x_i$ in $\Gamma_{\leq r}$.  The degree of the
  tropical meromorphic function $F|_{\Gamma_{\leq r}}$ on the metric graph
  $\Gamma_{\leq r}$ is zero, so we have
  \[ 0 = \sum_{y\in\Gamma_{\leq r}} \ord_y(F) = -\sum d_{v_{ij}}F(x_i) +
  \sum_{y\in\Gamma_{<r}}\ord_y(F), \]
  since $\ord_{v_{ij}}(F)$ is the \emph{incoming} slope.
  As each $-d_{v_{ij}} F(x_i)$ is nonnegative, we have
  \[ d_vF(x) \leq -\sum d_{v_{ij}} F(x_i) = -\sum_{y\in\Gamma_{<r}}\ord_y(F). \]
  Let $m_y = 2g(y) - 2 + \deg(y)$, the multiplicity of $y$ in $K_\Gamma$.
  Then $\ord_y(F) + m_y\geq 0$, so $-\ord_y(F)\leq m_y$ and hence,
  \begin{equation*}
    -\sum_{y\in\Gamma_{<r}}\ord_y(F) \leq \sum_{y\in\Gamma_{<r}}m_y \leq
    2g(\Gamma)-2,
  \end{equation*}
  since $K_\Gamma$ has degree $2g(\Gamma)-2$, and $m_y\geq 0$ for
  $y\notin\Gamma_{<r}$.

  Now we drop the assumption that $\Gamma$ has no genus-zero leaves.  Let $z$ be
  such a leaf, let $y$ be the first vertex along the edge adjoining $z$, that is, 
  the first point along this edge with $m_y \neq 0$, and let $e$ be the line
  segment joining $y$ and $z$.  The lemma is easy to prove when $\Gamma = e$, so
  we assume this is not the case.  Since $m_z = -1$, the incoming slope of $F$
  at $z$ is at least $1$.  From this it follows that $F$ is monotonically
  increasing from $y$ to $z$, and in particular, that the incoming slope of $F$
  at $y$ is at most $-1$.  Letting $\Gamma' = (\Gamma\setminus e)\cup\{y\}$,
  this implies that $F|_{\Gamma'}$ is a tropical meromorphic function satisfying
  $\div(F|_{\Gamma'}) + K_{\Gamma'}\geq 0$.

  Let $x\in \Gamma$. By repeatedly removing genus-zero leaf edges not containing
  $x$, we may find a subgraph $\Gamma''\subseteq \Gamma$ containing $x$ with at
  most one genus-zero leaf edge $e$ (which then contains $x$ by construction) such that
  $F|_{\Gamma''}$ is a tropical meromorphic function satisfying
  $\div(F|_{\Gamma''}) + K_{\Gamma''}\geq 0$.  Note that $g(\Gamma'')=g(\Gamma)$
  and that $\Gamma'$ has at most one point $z$ with $m_z < 0$.  If there is no
  such point, then the conclusion follows from the special case
  above. Otherwise, we proceed as before. Because $m_z=-1$, we obtain
  \[ -\sum_{y\in(\Gamma'')_{<r}}\ord_y(F) \leq \sum_{y\in(\Gamma'')_{<r}}m_y \leq
  2g(\Gamma'')-2+1 = 2g(\Gamma)-1. \]
\end{proof}

\subsection{Bounding zeros on wide opens}
\label{sec:bounding.zeros}
Let $U$ be a basic wide open subdomain of $X^\an$ with central point
$\zeta$, underlying affinoid $Y$, and annuli $A_1,\ldots,A_d$, as
in Section~\ref{sec:basic.wide.open}.  Suppose that $U$ is defined with respect to a
star neighborhood in a skeleton $\Gamma$.

\begin{definition}\label{def:thickness}
  The \defi{thickness} of $U$ is $\min\{a_1,\ldots,a_d\}$, where $a_i$ is the
  logarithmic modulus of $A_i$.
\end{definition}

Let $a$ be the thickness of $U$.  For $r\in(0,a)\cap\Q$, we let $U_r$ denote the
basic wide open subdomain inside of $U$ obtained by deleting a half-open annulus
of logarithmic modulus $r$ from each $A_i$, as in
Figure~\ref{fig:wideopen.bound}.

\begin{figure}[ht]
  \centering
  \begin{tikzpicture}[scale=1.5,
    every circle node/.style={draw, inner sep=.4mm},
    every node/.style={font=\tiny},
    other edge/.style={black!30!white, very thin}]
    \node (Z) at (0,0) [circle, fill=black, label=above:$\zeta$] {};
    \draw (Z) ++(30:1.5cm) 
          node (Z1) [circle, label=above:$\zeta_1$] {};
    \draw[other edge] (Z1) -- ++(0:.7cm)
          (Z1) -- ++(135:.7cm);
    \draw (Z) ++(180:1.7cm) 
          node (Z2) [circle, label=above:$\zeta_2$] {};
    \draw[other edge] (Z2) -- ++(135:.7cm)
          (Z2) -- ++(180:.7cm)
          (Z2) -- ++(225:.7cm);
    \draw (Z) ++(270:1.2cm) 
          node (Z3) [circle, label=below:$\zeta_3$] {};
    \draw[other edge] (Z3) -- ++(200:.7cm)
          (Z3) -- ++(340:.7cm);
    \path (Z) edge (Z1) (Z) edge (Z2) (Z) edge (Z3);
    \draw[very thick, ->] (Z1) -- ($(Z1)!.3cm!(Z)$)
          node [anchor=-90] {$v_1$};
    \draw[very thick, ->] (Z2) -- ($(Z2)!.3cm!(Z)$)
          node [anchor=70] {$v_2$};
    \draw[very thick, ->] (Z3) -- ($(Z3)!.3cm!(Z)$)
          node [anchor=150] {$v_3$};
    \node[circle, fill=white, label={[inner sep=0pt]135:$\zeta_1'$}] 
          (ZP1) at ($(Z1)!.6cm!(Z)$) {};
    \node[circle, fill=white, label=above:$\zeta_2'$] 
          (ZP2) at ($(Z2)!.6cm!(Z)$) {};
    \node[circle, fill=white, label={[inner sep=1pt]right:$\zeta_3'$}] 
          (ZP3) at ($(Z3)!.6cm!(Z)$) {};
    \path[densely dotted, very thick] 
          (Z) edge (ZP1) (Z) edge (ZP2) (Z) edge (ZP3);
    \draw[very thick, ->] (ZP1) -- ($(ZP1)!.3cm!(Z)$)
          node [anchor=135, inner sep=0pt] {$v_1'$};
    \draw[very thick, ->] (ZP2) -- ($(ZP2)!.3cm!(Z)$)
          node [anchor=70] {$v_2'$};
    \draw[very thick, ->] (ZP3) -- ($(ZP3)!.3cm!(Z)$)
          node [anchor=30, inner sep=3pt] {$v_3'$};
    \draw[|<->|, thin] ($(Z1)!.2cm!90:(Z)$) -- ($(ZP1)!.2cm!90:(Z)$)
          node [pos=.5, anchor=120] {$r$};
    \draw[|<->|, thin] ($(Z2)!-.3cm!90:(Z)$) -- ($(ZP2)!-.3cm!90:(Z)$)
          node [pos=.5, anchor=90] {$r$};
    \draw[|<->|, thin] ($(Z3)!.2cm!90:(Z)$) -- ($(ZP3)!.2cm!90:(Z)$)
          node [pos=.5, anchor=0] {$r$};
  \end{tikzpicture}
  
  \caption{Illustration of Definition~\ref{def:thickness} and the proof of
    Theorem~\ref{thm:wideopen.bounds}.  The lines of standard thickness
    represent an open star neighborhood $V$ of $\zeta$ such that
    $U = \tau^{-1}(V)$.  The dotted lines represent the smaller open star
    neighborhood $V_r$ of $\zeta$ such that $U_r = \tau^{-1}(V_r)$.  The ends of
    $U$ (resp., $U_r$) are $\zeta_1,\zeta_2,\zeta_3$ (resp.,
    $\zeta_1',\zeta_2',\zeta_3'$).}
  \label{fig:wideopen.bound}
\end{figure}
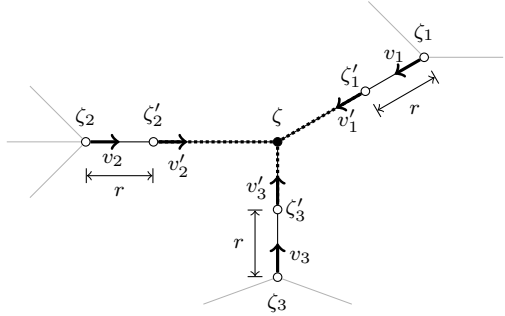

\begin{theorem}\label{thm:wideopen.bounds}
  With the above notation,
  let $\omega\in H^0(X,\Omega^1_{X/\C_p})$ be a nonzero global differential, and
  suppose that $\omega$ is exact on $U$, so $\omega = df$ for an analytic
  function $f$ on $U$.  Then $f$ has at most
  $\deg(\zeta)N_p(r,\,2g-1)$ geometric zeros, counted
  with multiplicity, on $U_r$, where  $\deg(\zeta)$ is the valency of $\zeta$
  in $\Gamma$.  If $U$ is defined with respect to a star neighborhood in a
  skeleton with no genus-zero leaves, then we may replace $2g-1$ by $2g-2$.
\end{theorem}

\begin{proof}
  Let $F_0 = -\log\|\omega\|$, and let $F = -\log|f|$, as in the statement of
  Proposition~\ref{p:annularbounds}.  As explained in
  Section~\ref{sec:canon.div.graph}, $F_0$ is a section of the tropical canonical
  bundle on $\Gamma$; that is, it satisfies the hypotheses of
  Lemma~\ref{lem:unstable.combinatorics}.  Hence, the absolute value of the slope
  of $F_0$ in any direction at any point of $\Gamma$ is at most $2g-1$, or
  $2g-2$ if $\Gamma$ has no genus-zero leaves.  In the latter case one may
  replace $2g-1$ by $2g-2$ everywhere below.

  Let $d = \deg(\zeta)$,
  let $\zeta_1,\ldots,\zeta_d$ (resp., $\zeta_1',\ldots,\zeta_d'$) be the ends
  of $U$ (resp., $U_r$), and let $v_i$ (resp., $v_i'$) be the tangent direction
  at $\zeta_i$ (resp., $\zeta_i'$) pointing in the direction of the central
  point $\zeta$, as in Figure~\ref{fig:wideopen.bound}.  
  By the above,
  we have $d_{v_i}F(\zeta_i)\leq 2g-1$ for all $i=1,\ldots,d$.  By
  Proposition~\ref{p:annularbounds}, 
  \[ d_{v_i'} F(\zeta_i') \leq N_p(r,\, d_{v_i}F(\zeta_i)) \leq N_p(r,\, 2g-1) \]
  for all $i=1,\ldots,d$.   By Proposition~\ref{prop:zeros.on.wide.open}, then,
  \[ \deg\big(\div(f|_{U_r})\big) = \sum_{i=1}^d d_{v_i'} F(\zeta_i')
  \leq d\,N_p(r,\,2g-1). 
  \]
\end{proof}

The following corollary plays the role
of~\cite[Proposition~7.7]{Stoll:uniformity}, with the slope bound of
Lemma~\ref{lem:unstable.combinatorics}
replacing~\cite[Corollary~6.7]{Stoll:uniformity}.

\begin{corollary}\label{cor:actual.annulus.bound}
  In the setting of Theorem~\ref{thm:wideopen.bounds}, if $U$ is an open annulus,
  then $f$ has at most $2\,N_p(r,\,2g-1)$ zeros on $U_r$.
\end{corollary}

By Lemma~\ref{lem:skel.from.annulus}, any open subdomain $U\subset X^\an$ which
is isomorphic to an open annulus has the form $\tau^{-1}(e)$ for an open edge
$e$ of some skeleton $\Gamma$ of $X$.  This is a basic wide open subdomain with
respect to the star neighborhood $e$ of the midpoint $\zeta$ of $e$.  Therefore
Corollary~\ref{cor:actual.annulus.bound} applies to any embedded open annulus.

\section{Uniform Bounds}
\label{sec:uniform-bounds}

In this section we use the following notation:

\begin{longtable}[l]{rl}
  \hspace{1em} $K$ & A local field of characteristic $0$. \\
  $\varpi$ & A uniformizer of $K$. \\
  $k$ & The residue field of $K$. \\
  $p$ & The characteristic of $k$. \\
  $q$ & The number of elements of $k$. \\
  $e$ & The ramification degree of $\sO_K$ over $\Z_p$. \\
  $X$ & A smooth, proper, geometrically connected curve over $K$. \\
  $g$ & The genus of $X$, assumed to be $\geq 2$. \\
  $J$ & The Jacobian of $X$. \\
  $\iota$ & $\colon X\injects J$, an Abel--Jacobi map defined over $K$. \\
\end{longtable}

\noindent
We normalize the valuation on $K$ such that $\val(p) = 1$, and we fix an
isometric embedding $K\injects\C_p$.  Recall that $N_p(\scdot,\scdot)$ is
defined in Definition~\ref{def:Np}.

\subsection{Uniform bounds on $K$-rational points}
In the following theorem we combine Corollary~\ref{cor:actual.annulus.bound}
and~\cite[Proposition~5.3]{Stoll:uniformity} to obtain uniform bounds on the
number of $K$-points of $X$ mapping into a subgroup of $J(K)$ of a given rank
$\rho$.  This generalizes~\cite[Theorem~9.1]{Stoll:uniformity}.

\begin{theorem}\label{thm:with.ss.model}
  Let $G\subset J(K)$ be a subgroup of rank $\rho\leq g-3$.  Then
  \[ \#\iota^{-1}(G) \leq \big(5qg + 6g - 2q - 8\big)\,N_p(1/e,\,2g-1). \]
\end{theorem}

\begin{proof}
  We will use $X^\an$ to denote the $\C_p$-analytic space
  $(X\tensor_K\C_p)^\an$.  Let $V$ be the annihilator in
  $H^0(X_{\C_p},\Omega^1_{X_{\C_p}/\C_p}) = H^0(J_{\C_p},\Omega^1_{J_{\C_p}/\C_p})$ of
  $\log_{J(\C_p)}(G)$, with the notation in Section~\ref{sec:abelian-integral}.
  By the standard Chabauty--Coleman calculation, $V$ has dimension at least
  $g-\rho\geq 3$.  Moreover, for $\omega\in V$ we have $\Abint_x^y\omega = 0$ for
  all $x,y\in\iota^{-1}(G)$.

  Let $B\subset X^\an$ be an open subdomain defined over $K$ which is
  $K$-isomorphic to the open unit disc $\bB(1)_+$.  Suppose that there exists
  $x\in X(K)\cap B$.  There is an isomorphism $T\colon B\isomto\bB(1)_+$ such
  that $x\mapsto 0$ and $B\cap X(K)$ is identified with $\varpi\sO_K$.  In
  particular, $B\cap X(K)\subset B_r\coloneq \{\eta\in B:-\log|T(\eta)| > r\}$
  for all $r < \frac 1e$.  For any nonzero $\omega\in V$, there is a unique
  analytic function $f$ on $B$ such that $df = \omega$ and $f(x) = 0$.  For
  $y\in B(\C_p)$, we have $\Abint_x^y\omega = \BCint_x^y\omega = f(y)$.  Hence
  points of $\iota^{-1}(G)$ in $B$ are zeros of $f$, so
  \begin{equation}\label{eq:chabauty.on.balls}
    \#\iota^{-1}(G)\cap B\leq \lim_{r\nearrow 1/e} N_p(r,\,(2g-2)+1)
    = N_p(1/e,\, 2g-1)
  \end{equation}
  by Corollary~\ref{cor:disc.bound}.

  Now let $A\subset X^\an$ be an open subdomain defined over $K$ which is
  $K$-isomorphic to an open annulus $\bS(\varpi^b)_+$ for $b\geq 1$.  Then
  $A\cap X(K)\subset A_r\coloneq \{\eta\in A:-\log|T(\eta)|\in(r,b/e-r)\}$ for
  all $r < \frac 1e$ as above. Suppose that there exists $x\in A\cap X(K)$.
  This implies that $b\geq 2$.  Choose $\omega\in V$ nonzero which is exact on $A$
  and such that $\BCint_\gamma\omega = \Abint_\gamma\omega$ for all paths
  $\gamma$.  This is possible because both are codimension-one conditions on
  $\omega$: namely, that $\Res(\omega) = 0$ in the notation
  of Section~\ref{sec:de.rham} and that $a(\omega) = 0$ in the notation of
  Proposition~\ref{prop:correction.annulus}.  As above, there is an analytic
  function $f$ on $A$ such that $df = \omega$ and all points of $\iota^{-1}(G)$
  in $A$ are zeros of $f$.  By Corollary~\ref{cor:actual.annulus.bound}, then,
  \begin{equation}\label{eq:chabauty.on.annuli}
    \#\iota^{-1}(G)\cap A\leq \lim_{r\nearrow 1/e} 2N_p(r,\,2g-1) = 2N_p(1/e,\, 2g-1).
  \end{equation}

  By~\cite[Proposition~5.3]{Stoll:uniformity}, there exists
  $t\in\{0,1,2,\ldots,g\}$ such that $X(K)$ is covered by at most
  $(5q+2)(g-1)-3q(t-1)$ embedded open discs and at most embedded $2g-3+t$ open
  annuli, all defined over $K$.  Using~\eqref{eq:chabauty.on.balls}
  and~\eqref{eq:chabauty.on.annuli}, then, we have
  \[\begin{split}
    \#\iota^{-1}(G) &\leq \big((5q+2)(g-1)-3q(t-1)\big)\,N_p(1/e,\,2g-1)
    + 2(2g-3+t)\,N_p(1/e,\,2g-1) \\
    &\leq \big((5q+2)(g-1)-3q(t-1)+4g-6+2t)\big)\,N_p(1/e,\,2g-1) \\
    &\leq (5qg + 6g - 2q - 8)\,N_p(1/e,\,2g-1),
  \end{split}\]
  where the third inequality holds because the quantity is maximized at $t=0$.
\end{proof}

Suppose now that $X$ is defined over a number field $F$.  Let $\fp$ be a prime
of $F$ over $2$, and let $K = F_\fp$.  The number $q$ of elements of the residue
field $k$ of $F_\fp$ and the ramification degree of $F_\fp$ over $\Z_2$ are both
bounded in terms of the degree $[K:\Q]$.  Applying
Theorem~\ref{thm:with.ss.model} with $G = J(F)$ yields
Theorem~\ref{T:uniformity-K-points}, and applying
Theorem~\ref{thm:with.ss.model} with $G = J(F)_\tors$ yields
Theorem~\ref{T:uniformity-K-torsion}. 

\begin{remark}
  It should be possible to refine the bound of Theorem~\ref{thm:with.ss.model}
  to include the rank $\rho$, as in~\cite[Theorem~8.1]{Stoll:uniformity}, although
  it is not obvious how to generalize Corollary~\ref{cor:actual.annulus.bound}
  in this way.
\end{remark}

\subsection{Uniform bounds on geometric torsion packets}
In the following theorem, the Abel--Jacobi map $\iota:X\injects J$ need only be
defined over $\C_p$.  The requirement that $X$ be defined over $K$ and not just
over $\C_p$ is only used to bound from below the minimum length of an edge in
a skeleton $\Gamma$; the resulting bounds depend on $K$ only through its ramification
degree over $\Z_p$.  We set
\[ E(g,p) \coloneq
\begin{cases}
  \#\GSp_{2g}(\F_5) &\text{if } p\neq 5 \\
  \#\GSp_{2g}(\F_7) &\text{if } p = 5.
\end{cases}\]
Note that
\[
\#\GSp_{2g}(\F_{\ell}) = \left({\ell}^{2g}-1\right) \left({\ell}^{2g-2}-1\right) \cdots
\left({\ell}^{2}-1\right) \cdot {\ell}^{g^2}\cdot \left({\ell-1}\right) <
\ell^{2g^2 + g + 1}
\]
for any prime $\ell$.

\begin{theorem}\label{thm:geom.torsion.bound}
  Let $\Gamma$ be the minimal skeleton of $X_{\C_p}$, considered as a
  vertex-weighted metric graph.
  \begin{enumerate}
  \item If $g > 2g(v) + \deg(v)$ for all vertices $v$ of $\Gamma$, then 
    \[ \#\iota^{-1}(J(\C_p)_\tors) \leq (16g^2-12g)\,N_p\big((4eE(g,p))^{-1},\,2g-2\big). \]
  \item If $g > 2g(v) + 2\deg(v) - 2$ for all vertices $v$ of $\Gamma$, then 
    \[ \#\iota^{-1}(J(\C_p)_\tors) \leq (8g-6)\,N_p\big((4eE(g,p))^{-1},\,2g-2\big). \]
  \end{enumerate}
\end{theorem}

Note that the bounds only depend on $p$ through the correction factor
$N_p(\scdot,\scdot)$, which can be removed by recalling that
$N_2(\scdot,\scdot)\geq N_p(\scdot,\scdot)$.

\begin{proof}
  First suppose that $X$ admits a split stable model $\fX$ over $\sO_{K}$, so
  that $\Gamma = \Gamma_\fX$.  The hypotheses imply that $X$ does not have good
  reduction, namely, that $\Gamma$ is not a point.  Let $\Gamma'$ denote the
  metric graph obtained from $\Gamma$ by adding a vertex at the midpoint of each
  loop edge.  Since $\Gamma$ is stable, $2g(v) + \deg(v)\geq 3$ for all vertices
  $v$ of $\Gamma$, and $2g(v)+\deg(v)=2$ if $v$ is a midpoint of a loop edge, so
  $g > 2g(v)+\deg(v)$ for all vertices of $\Gamma'$.  Note that $\Gamma'$ has at
  most $3g-2$ vertices and $4g-3$ edges by Lemma~\ref{lem:graph.combinatorics}.
  Since our model $\fX$ is split, each edge of $\Gamma'$ has length at least
  $1/2e$.

  For each vertex $v$ of $\Gamma'$, let $S_v$ denote the union of $v$ and all
  open edges adjacent to $v$, and let $U_v = \tau^{-1}(S_v)$.  Then $U_v$ is a
  basic wide open subdomain of $X^\an$ of thickness
  (Definition~\ref{def:thickness}) at least $1/2e$.  By Theorem~\ref{thm:H1dR}
  the space $V_v\subset H^0(X_{\C_p},\Omega^1_{X_{\C_p}/\C_p})$ of $1$-forms
  $\omega$ which are exact on $U_v$ has dimension at least
  \[ \dim(V_v) \geq g - (2g(v)-1+\deg(v)) \geq 2. \]
  Let $\epsilon$ be an open edge of $\Gamma'$ adjacent to $v$, and let
  $U_{v,\epsilon} = \tau^{-1}(\{v\}\cup \epsilon)\subset U_v$, the union of the
  underlying affinoid of $U_v$ with the open annulus $\tau^{-1}(\epsilon)$.  By
  Proposition~\ref{prop:correction.annulus}, there exists a nonzero differential
  $\omega\in V_v$ such that $\BCint_x^y\omega = \Abint_x^y\omega$ for all
  $x,y\in U_{v,\epsilon}(\C_p)$.

  Suppose that there exists $x_0\in U_{v,\epsilon}(\C_p)$ such that
  $\iota(x_0)\in J(\C_p)_\tors$.  Since $\omega$ is exact, we have $\omega = df$
  for an analytic function $f$ on $U_{v}$ such that $f(x_0) = 0$.  Since
  \[ f(y) = \BCint_{x_0}^y\omega = \Abint_{x_0}^y\omega =
  \angles{\log_{J(\C_p)}(\iota(y)-\iota(x_0)),\,\omega} \]
  for $y\in U_{v,\epsilon}(\C_p)$ and since $\log_{J(\C_p)}$ vanishes on
  $J(\C_p)_\tors$, we have $f(y) = 0$ for all
  $y\in\iota^{-1}(J(\C_p)_\tors)\cap U_{v,\epsilon}$.  Choose $r\in(0,1/4e)$, define
  $U_{v,r}\subset U_v$ as in Section~\ref{sec:bounding.zeros}, and let
  $U_{v,\epsilon,r} = U_{v,r}\cap U_{v,\epsilon}$.  Then
  \begin{equation}\label{eq:Uver.bound}
    \#\big(\iota^{-1}(J(\C_p)_\tors)\cap U_{v,\epsilon,r}\big)
    \leq \deg(\div(f|_{U_{v,\epsilon,r}})) \leq \deg(\div(f|_{U_{v,r}}))
    \leq 2g\, N_p(r,\,2g-2),
  \end{equation}
  where we have used Theorem~\ref{thm:wideopen.bounds} and
  Lemma~\ref{lem:graph.combinatorics}(3) for the final inequality.
  We have $X^\an = \bigcup_{v,\epsilon} U_{v,\epsilon,r}$, where the union is
  taken over all vertices $v$ of $\Gamma'$ and all open edges $\epsilon$
  adjacent to $v$, and where $r\in(0,1/4e)$ (recall that $1/4e$ is half the
  minimum length of an edge).  The number of pairs $(v,\epsilon)$ consisting of a
  vertex and an adjacent edge is equal to twice the number of edges, which is at
  most $8g-6$.  Therefore,
  \[ \#\iota^{-1}(J(\C_p)_\tors) \leq (8g-6)(2g)N_p(r,\,2g-2) \]
  for all $r < 1/4e$.  Taking the limit as $r\nearrow 1/4e$ yields assertion~(1)
  in this case.

  Now suppose that $g > 2g(v) + 2\deg(v) - 2$ for all vertices $v$ of $\Gamma$
  (hence of $\Gamma'$).  Then $\dim(V_v)\geq \deg(v)$, so by
  Proposition~\ref{prop:correction.wide.open}, there exists a nonzero
  differential $\omega\in V_v$ such that $\BCint_x^y\omega = \Abint_x^y\omega$
  for all $x,y\in U_{v}(\C_p)$.  Proceeding as above, we see that
  \[ \#\big(\iota^{-1}(J(\C_p)_\tors)\cap U_{v,r}\big) \leq
  \deg(v)N_p(r,\,2g-2) \]
  for all $r\in(0,1/4e)$.  Using the facts that $X^\an = \bigcup_v U_{v,r}$
  and that $\sum_v\deg(v)$ is twice the number of edges in $\Gamma'$, we have
  \[ \#\iota^{-1}(J(\C_p)_\tors) \leq
  \sum_v \deg(v)N_p(r,\,2g-2) \leq (8g-6)N_p(r,\,2g-2).
  \]
  Taking the limit as $r\nearrow 1/4e$ completes the proof in this case.

  Finally, we reduce to the case when $X$ admits a split stable model over $K$ by
  making a potentially ramified field extension $K'/K$.  
  By \cite[Theorem~2.4]{DeligneMumford:irreducibility}, $X$ admits a stable
  model over $\sO_K$ if and only if its Jacobian $J$ has stable reduction, that is, if
  and only if the connected component of the special fiber of the N\'eron model
  of $J_K$ is semiabelian.
  By~\cite[Corollary~6.3]{silverbergZarhin:semistableReduction}, for any prime
  $\ell \geq 5$ which is coprime to $p$, if $K'' = K(J[\ell])$, then $J_{K''}$
  admits a stable model.  Since $J$ is principally polarized,
  $\Gal(K''/K) \subset \GSp_{2g}(\F_{\ell})$.
  Choosing $\ell = 5$ or, if $p = 5$, $\ell = 7$, gives
  $[K'': K] \leq E(g,p)$.  In particular, the ramification degree of $K''/K$ is at
  most $E(g,p)$, so the ramification degree of $K''/\Z_p$ is at most $eE(g,p)$.
  The stable model of $X_{K''}$ may not be split, but it can be made split by
  trivializing the action of $\Gal(\bar k/k)$ on the geometric skeleton
  $\Gamma$.  This results in an unramified extension $K'$ of $K''$.  Now we 
  apply the above argument to the curve $X_{K'}$.
\end{proof}

\begin{remark}
  The hypotheses of Theorem~\ref{thm:geom.torsion.bound}(1) are satisfied if $X$
  is a Mumford curve of genus $g$ and all vertices of $\Gamma$ have valency at
  most $g-1$, namely, if $g\geq 4$ and $\Gamma$ is trivalent.  The hypotheses of
  Theorem~\ref{thm:geom.torsion.bound}(2) are satisfied if $X$ is a Mumford
  curve of genus $g$ and all vertices of $\Gamma$ have valency at most
  $g/2$, namely, if $g\geq 6$ and $\Gamma$ is trivalent.
\end{remark}

\subsection*{Acknowledgements}

This work clearly owes a debt to Stoll~\cite{Stoll:uniformity}, who, in addition
to being an inspiration, had many helpful comments on an early draft of this
paper.  The authors would also like to thank Matt Baker, Vladimir Berkovich,
Kiran Kedlaya, Dino Lorenzini, Andrew Obus, Jennifer Park, Bjorn Poonen, Alice
Silverberg, and Yuri Zarhin for helpful discussions; Matt Baker and Walter
Gubler for further comments on an early draft; and Christian Vilsmeier for
important corrections.  The authors thank the referees for a number of comments
and corrections.  Katz was supported by a National Sciences and Energy Research
Council Discovery grant.  Rabinoff was sponsored by a National Security Agency
Young Investigator grant.
Zureick-Brown was supported by a National Security Agency Young Investigator
grant.

\bibliography{KRZB_tropical_uniformity}
\bibliographystyle{amsalpha}

\end{document}